\documentclass[11pt]{amsart}
\usepackage{geometry}
\usepackage{graphicx}
\usepackage{amssymb}
\usepackage{cmll} 

\usepackage{amsthm} 
\usepackage{natbib} 
\usepackage[pdftex,hidelinks]{hyperref} 
\usepackage[acronym]{glossaries} 
\usepackage{tikz} 
\usetikzlibrary{shapes} 
\usepackage{subcaption} 

\usepackage{algorithm} 
\usepackage{algpseudocode}

\newtheorem{lmm}{Lemma}
\newtheorem{crl}{Corollary}
\newtheorem{thm}{Theorem}
\theoremstyle{definition}
\newtheorem{xmp}{Example}
\newtheorem{dfn}{Definition}

\newacronym{dag}{DAG}{directed acyclic graph}
\newacronym{mdag}{mDAG}{marginal directed acyclic graph}
\newacronym{bn}{BN}{Bayesian network}
\newacronym{mag}{MAG}{maximal ancestral graph}
\newacronym{ci}{CI}{conditional independence}
\newacronym{pag}{PAG}{partial ancestral graph}
\newacronym{fci}{FCI}{Fast Causal Inference}
\newacronym{pdag}{PDAG}{partially directed acyclic graph}
\newacronym{cpdag}{CPDAG}{completed partially directed acyclic graph}
\newacronym{admg}{ADMG}{acyclic directed mixed graph}
\newacronym{dfs}{DFS}{depth-first search}

\usepackage{xparse}

\NewDocumentCommand\ind{mmg}{%
  \IfNoValueTF{#3}
    {#1 \Perp #2}
    {#1 \Perp #2 \mid #3}%
}

\NewDocumentCommand\indin{mmmg}{%
  \IfNoValueTF{#4}
    {#2 \Perp_{#1} #3}
    {#2 \Perp_{#1} #3 \mid #4}%
}

\newcommand{\adjin}[2]{\mathbf{ADJ}_{#1}(#2)}

\newcommand{\pain}[2]{\mathbf{PA}_{#1}(#2)}
\newcommand{\chin}[2]{\mathbf{CH}_{#1}(#2)}
\newcommand{\anin}[2]{\mathbf{AN}_{#1}(#2)}
\newcommand{\dein}[2]{\mathbf{DE}_{#1}(#2)}
\newcommand{\companin}[2]{\mathbf{compAN}_{#1}(#2)}

\newcommand{\painvalues}[2]{\mathbf{pa}_{#1}(#2)}

\newcommand{\painfixedvalues}[2]{\widehat{\mathbf{pa}}_{#1}(#2)}

\newcommand{\undirectededge}[1][.7em]{\mathrel{\rule[.5ex]{#1}{.5pt}}}

\newcommand{\nonadjacencyin}[3]{\neg \text{Adj}_{#1}(#2, #3)}

\DeclareMathOperator{\bnmodel}{\mathbf{BNM}}

\newenvironment{proofln}[1][\proofname]{%
  \begin{proof}[#1]$ $\par\nobreak\ignorespaces
}{%
  \end{proof}
}

\title{Towards Characterising Bayesian Network Models under Selection}

\author{Angelos P. Armen \and Robin J. Evans}
\date{\today}

\begin{document}

\begin{abstract}
Real-life statistical samples are often plagued by selection bias, which complicates drawing conclusions about the general population. When learning causal relationships between the variables is of interest, the sample may be assumed to be from a distribution in a causal Bayesian network (BN) model under selection. Understanding the constraints in the model under selection is the first step towards recovering causal structure in the original model. The conditional-independence (CI) constraints in a BN model under selection have been already characterised; there exist, however, additional, non-CI constraints in such models. In this work, some initial results are provided that simplify the characterisation problem. In addition, an algorithm is designed for identifying compelled ancestors (definite causes) from a completed partially directed acyclic graph (CPDAG). Finally, a non-CI, non-factorisation constraint in a BN model under selection is computed for the first time.
\end{abstract}

\maketitle

\section{Introduction}

Real-life statistical samples are often not from the population of interest but from a subpopulation with fixed values for a set of \emph{selection variables}. The most prominent example is case--control studies. Suppose that $S$ is a boolean variable indicating whether an individual is included (``selected'') in the study. Then a case--control sample with variables $\mathbf{V}$ may be assumed to be a random sample from the conditional distribution over $\mathbf{V}$ given $S = \text{true}$ \citep{spirtes2000causation}. We refer to conditioning on fixed values of a set of variables as \emph{selection}. Selection may create ``spurious'' dependencies, that is, dependencies that are not present in the general population, as the following example \citep{pearl2009causality} shows.

\begin{xmp}
Suppose that admission to a certain college requires either high grades or musical talent. Even if grades and musical talent are uncorrelated in the general population, they will be negatively correlated in the college population: students with low grades will be most likely musically gifted, which would explain their admission to the college, and, accordingly, students with no musical talent will probably have high grades. 
\end{xmp}

This phenomenon is known as \emph{selection bias}, \emph{Berkson's paradox} \citep{berkson1946limitations}, and the \emph{explaining away effect} \citep{kim1983computational} (in the example above, one explanation for admission renders the other one less likely).

Applying algorithms that assume a random sample in their input to selection-biased data may lead to incorrect output. Thus, several approaches have been developed to deal with selection bias in various tasks. The motivation behind this work is to improve causal structure learning from selection-biased data.

Elucidating causal relationships is of utmost importance in science. A \emph{causal \gls{dag}} represents the direct causal relationships between a set of variables (``direct'' meaning not through other variables in the set). In the absence of hidden common causes, causal feedback loops, and selection bias, the causal \gls{dag} also represents the \glspl{ci} between the variables based on the so-called \emph{Markov condition} \citep{neapolitan2004learning}; the \emph{causal \gls{bn} model} is the set of distributions that satisfy the \gls{ci} constraints encoded by the causal \gls{dag}. By performing hypothesis tests of \gls{ci} on a sample from the probability distribution over the variables, \emph{constraint-based structure learning} algorithms such as PC \citep{spirtes1999algorithm} can learn the causal \gls{bn} model, which amounts to learning (features of) the causal \gls{dag}. In the presence of selection bias (and/or hidden common causes), however, the probability distribution over the variables may no longer be in the causal \gls{bn} model and not all constraints on the distribution are \gls{ci} constraints; in that case, the PC algorithm is not appropriate.

Some previous work focussed on learning causal \gls{bn} models from selection-biased samples by correcting for the selection bias. \citet{cooper2000bayesian} devised a Bayesian method in which the biased sample is treated as a random sample with missing values for a known or unknown number of unsampled individuals, and the likelihood of the data is computed by summing over all possibilities for the missing values and the number of unsampled individuals, if unknown. This approach is computationally intractable in all but the smallest examples. In addition, it requires knowledge of the non-random-sampling process. \citet{borboudakis2015bayesian} devised a \gls{ci} test for case--control samples with categorical variables, characterised potentially spurious links when learning the skeleton of the causal \gls{dag} using a test for random samples, and proposed the use of their specialised test on these links as a post-processing approach to removing spurious links. The drawbacks of this approach is that it is not applicable to general selection-biased samples, the joint distribution over the selection variables needs to be known, and the specialised tests are less powerful than the ones for random samples.

A more general approach to deal with the problem of selection bias is to characterise (supermodels of) \gls{bn} models under selection and design algorithms for learning those models.
 
 An \emph{ordinary Markov model} is defined by the \gls{ci} constraints in a \gls{bn} model under marginalisation and/or selection \citep{shpitser2014introduction}. A \emph{\gls{mag}} \citep{richardson2002ancestral} is a graphical representation of an ordinary Markov model that can be learned using an algorithm such as FCI \citep{spirtes1999algorithm,zhang2008completeness}. A \gls{mag} also represents ancestral relationships in the \gls{dag} of a \gls{bn} model; in the case of a causal \gls{bn} model, these are (indirect) causal relationships.

There are, however, additional, non-\gls{ci} constraints in a \gls{bn} model under marginalisation and/or selection. A \emph{nested Markov model} \citep{richardson2017nested} is defined by the \gls{ci} constraints and the so-called \emph{Verma constraints} \citep{robins1986new} in a \gls{bn} model under marginalisation. These constraints may be represented by either a \emph{\gls{mdag}} \citep{evans2016graphs} or an \emph{\gls{admg}} \citep{richardson2017nested} and comprise all the equality constraints in the margin of a \gls{bn} model over a set of categorical variables \citep{evans2018margins}, although there there also exist inequality constraints such as \emph{Bell's inequalities} \citep[see, for example,][for details]{wolfe2016inflation}. In contrast to \glspl{mag}, \glspl{mdag} and \glspl{admg} represent direct causal relationships, and are therefore more expressive. In the case of selection, there exist non-\gls{ci} equality constraints as well. \citet{lauritzen1999generating} showed that there exist non-\gls{ci} \emph{factorisation} constraints. \citet{evans2015recovering} showed the existence of non-\gls{ci}, non-factorisation constraints in a certain categorical \gls{bn} model under conditioning; here we show that those constraints are also constraints in another \gls{bn} model under selection and explicitly identify the sole constraint in the case of binary variables.

In this work, some initial results are provided that simplify the problem of characterising \gls{bn} models under selection. For the sake of simplicity, the results are presented for categorical variables, but they can be easily generalised to general state spaces where a joint density with respect to a product measure exists. All proofs can be found in the appendix of this paper.

Having characterised equality constraints in \gls{bn} models under selection, it may be possible to devise a graphical representation of them. Structure-learning algorithms for selection-biased data can then be developed that make use of the non-\gls{ci} information in the data and potentially enable the learning of more causal relationships than is currently possible.

\section{Background}

In this paper, sets are in boldface (e.g., $\mathbf{S}$), $X$ is used a shortcut for the singleton $\{X\}$ in places were a set is expected, and $\mathbf{A} \dot{\cup} \mathbf{B}$ denotes the union of disjoint sets $\mathbf{A}$ and $\mathbf{B}$. Random variables are denoted by capital letters (e.g., $X$) and their values by the respective lowercase letters (e.g., $x$); fixed values are denoted by a hat (e.g., $\hat{x}$). If $\mathbf{X}$ and $\mathbf{Y}$ are random variables, $\mathbf{x} \cup \mathbf{y}$ denotes the values of $\mathbf{X} \cup \mathbf{Y}$; the same is true for set operations other than the union. Probability distributions are denoted by capital letters (e.g., $P$) and their respective probability density functions by the respective lowercase letters (e.g., $p$); $p(\mathbf{x})$ is used as a shortcut for $p(\mathbf{X} = \mathbf{x})$. If $P$ is a distribution over $\mathbf{V}$, $\mathbf{X}$, $\mathbf{Y}$, and $\mathbf{Z}$ are distinct subsets of $\mathbf{V}$, and $\mathbf{X}$ and $\mathbf{Y}$ are nonempty, the conditional independence of $\mathbf{X}$ and $\mathbf{Y}$ given $\mathbf{Z}$ in $P$ is denoted by $\indin{P}{\mathbf{X}}{\mathbf{Y}}{\mathbf{Z}}$. If $P$ is a distribution over $\mathbf{O} \dot{\cup} \mathbf{H} \dot{\cup} \mathbf{C} \dot{\cup} \mathbf{S}$, then the marginal/conditional distribution of $P$ over $\mathbf{O}$ given $\mathbf{C}$ and $\mathbf{S} = \hat{\mathbf{s}}$ is denoted by $P[_{\mathbf{H}}^{\mathbf{C},\mathbf{S} = \hat{\mathbf{s}}}$. Finally, graphs are in calligraphic (e.g., $\mathcal{G}$).

A \emph{graph} $\mathcal{G}$ is an ordered pair $(\mathbf{V}, \mathbf{E})$ of a set of nodes $\mathbf{V}$ and a set of edges $\mathbf{E}$ that connect pairs of distinct nodes in $\mathbf{V}$. If there is an edge between nodes $X$ and $Y$ in $\mathcal{G}$, then $X$ and $Y$ are \emph{adjacent} in $\mathcal{G}$. The union of the sets of nodes adjacent to nodes $\mathbf{X}$ in $\mathcal{G}$ is denoted by $\adjin{\mathcal{G}}{\mathbf{X}}$. A sequence of $n \ge 2$ nodes $(X_1, \ldots, X_n)$ such that, for $2 \le i \le n$, $X_{i - 1}$ and $X_i$ are adjacent, is called a \emph{path} from $X_1$ to $X_n$. $\mathcal{G}$ is \emph{connected} if there is a path from every node to every other node in the graph.
Let $p = (X_1, \ldots, X_n)$ be a path. The nodes $X_2, \ldots, X_{n - 1}$ are called \emph{interior} nodes on $p$. Path $(X_i, \ldots, X_j)$, where $1 \le i < j \le n$, is the \emph{subpath} of $p$ from $X_i$ to $X_j$ and is denoted by $p(X_i, Y_j)$. If $X_1 = X_n$, $p$ is a \emph{cycle}; if also $X_1, \ldots, X_{n - 1}$ are distinct, $p$ is a \emph{simple cycle}. A path is \emph{simple} if no subpath is a cycle. A \emph{triple} is a simple path with three nodes. A triple $(X, Z, Y)$ is \emph{shielded} if $X$ and $Y$ are adjacent. A simple path is shielded if there is a shielded triple on the path. Let $\mathcal{G}_1 = (\mathbf{V}_1, \mathbf{E}_1)$ and $\mathcal{G}_2 = (\mathbf{V}_2, \mathbf{E}_2)$ be graphs. $\mathcal{G}_1$ is a \emph{subgraph} of $\mathcal{G}_2$ (denoted by $\mathcal{G}_1 \subseteq \mathcal{G}_2$) if $\mathbf{V}_1 \subseteq \mathbf{V}_2$ and $\mathbf{E}_1 \subseteq \mathbf{E}_2$. The \emph{induced subgraph} of $\mathcal{G}$ over $\mathbf{A} \subseteq \mathbf{V}$ (denoted by $\mathcal{G}_{\mathbf{A}}$) is the graph with set of nodes $\mathbf{A}$ and the edges in $\mathcal{G}$ between nodes in $\mathbf{A}$.

A graph is called \emph{directed} (resp. \emph{undirected}) when its edges are directed (resp. undirected). A \emph{tree} is a connected undirected graph without (simple) cycles. In a tree, a \emph{leaf} is a node which is adjacent to a single node. If there is an edge $X \rightarrow Y$ in directed graph $\mathcal{G}$, then $X$ is a \emph{parent} of $Y$ and $Y$ a \emph{child} of $X$ in $\mathcal{G}$; the edge is said to be \emph{out of} $X$ and \emph{into} $Y$. The union of the sets of parents (resp. children) of nodes $\mathbf{X}$ in $\mathcal{G}$ is denoted by $\pain{\mathcal{G}}{\mathbf{X}}$ (resp. $\chin{\mathcal{G}}{\mathbf{X}}$). The set $X \cup \pain{\mathcal{G}}{X}$ is called the \emph{family} of $X$ in $\mathcal{G}$. A node without children is called a \emph{source} (resp. \emph{sink}). A \emph{link} is an edge without regard of direction, and the \emph{skeleton} of a directed graph is the undirected graph whose edges corresponds to links in the directed graph. A triple $(X, Y, Z)$ such that $X \rightarrow Z \leftarrow Y$ is a \emph{collider}. A path from $X$ to $Y$ is out of (resp. into) $X$ and out of (resp. into) $Y$ if the first edge of the path is out of (resp. into) $X$ and the last edge is out of (resp. into) $Y$.  A simple path from $X$ to $Y$ where all edges are directed towards $Y$ is called \emph{directed}. If there is a directed path from $X$ to $Y$ or $X = Y$, then $X$ is an \emph{ancestor} of $Y$ and $Y$ a \emph{descendant} of $X$. The union of the sets of ancestors (resp. descendants) of nodes $\mathbf{X}$ in $\mathcal{G}$ is denoted by $\anin{\mathcal{G}}{\mathbf{X}}$ (resp. $\dein{\mathcal{G}}{\mathbf{X}}$). A simple cycle $(X_1, \ldots, X_n)$ is \emph{directed} if for $2 \le i \le n$, the edge between $X_{i - 1}$ and $X_i$ is directed towards $X_i$. A \emph{\gls{dag}} is a directed graph without directed cycles. A \emph{conditional \gls{dag}} is a \gls{dag} ($\mathbf{X} \dot{\cup} \mathbf{Y}$, $\mathbf{E}$) such that the nodes in $\mathbf{Y}$ are sources \citep{evans2018margins}. The nodes in $\mathbf{X}$ and $\mathbf{Y}$ are called \emph{random nodes} and \emph{fixed nodes}, respectively; fixed nodes are drawn in a rectangle. The result of \emph{fixing} $\mathbf{A} \subseteq \mathbf{V}$ in \gls{dag} $\mathcal{G} = (\mathbf{V}, \mathbf{E})$ (denoted by $\phi_{\mathbf{A}}(\mathcal{G})$) is the conditional \gls{dag} with random nodes $\mathbf{V} \setminus \mathbf{A}$, fixed nodes $\mathbf{A}$, the edges in $\mathcal{G}$ between nodes in $\mathbf{V} \setminus \mathbf{A}$, and the edges in $\mathcal{G}$ from nodes in $\mathbf{A}$ to nodes in $\mathbf{V} \setminus \mathbf{A}$. A \emph{\gls{pdag}} is a partially directed graph without directed cycles.

Let $\mathbf{X}$ and $\mathbf{Y}$ be distinct sets of variables. A \emph{(conditional) model} over $\mathbf{X}$ given $\mathbf{Y}$ is a set of (conditional) probability distributions over $\mathbf{X}$ given $\mathbf{Y}$. 

\begin{dfn}[\gls{bn} model]
	Let $\mathbf{X}$ be a set of variables and $\mathcal{G}$ be a \gls{dag} over $\mathbf{X}$. The \emph{\gls{bn} model} defined by $\mathcal{G}$, denoted by $\bnmodel(\mathcal{G})$, is the set of distributions $P$ of $\mathbf{X}$ that satisfy the \emph{Markov condition} with $\mathcal{G}$, that is, every variable $X$ in $\mathbf{X}$ is independent of its non-descendants and non-parents in $\mathcal{G}$ given its parents in $\mathcal{G}$:
	\begin{displaymath}
		\indin{P}{X}{\mathbf{X} \setminus (\dein{\mathcal{G}}{X} \cup \pain{\mathcal{G}}{X})}{\pain{\mathcal{G}}{X}}
	\end{displaymath}
$\mathcal{G}$ is referred to as the \emph{structure} of $\bnmodel(\mathcal{G})$.
\end{dfn}

Let $\mathbf{X}$ be a set of categorical variables, $\mathcal{G}$ be a \gls{dag} over $\mathbf{X}$, and $P \in \bnmodel(\mathcal{G})$. Then $P$ equals the product of the conditional distributions of the nodes in $\mathcal{G}$ given their parents in $\mathcal{G}$: \citep[Theorem 1.4]{neapolitan2004learning}
\begin{displaymath}
	p(\mathbf{x}) = \prod_{X \in \mathbf{X}} p(x \mid \painvalues{\mathcal{G}}{X})
\end{displaymath}

Suppose that conditional distributions of the nodes in $\mathcal{G}$ given their parents in $\mathcal{G}$ are specified. Then the product of the distributions is in $\bnmodel(\mathcal{G})$ \citep[Theorem 1.5]{neapolitan2004learning}.

\gls{bn} models are widely used for causal inference. A \emph{causal \gls{bn} model} is defined by a \emph{causal \gls{dag}}, whose edges denote direct causal relationships between the variables \citep[See][for the exact definition of (direct) causation used]{neapolitan2004learning}. In the absence of hidden common causes, causal feedback loops, and selection bias, the distribution over the variables in a causal \gls{dag} $\mathcal{G}$ is in $\bnmodel(\mathcal{G})$ \citep{neapolitan2004learning}. Therefore, structure learning from a sample amounts to learning direct causal relationships between the variables in the sample.

Different \glspl{dag} may impose the same factorisation on the distribution over the variables and, therefore, impose the same \gls{ci} constraints and define the same \gls{bn} model; these \glspl{dag} are called \emph{Markov equivalent} and said to belong to the same \emph{Markov equivalence class}. Two \glspl{dag} are in the same class if and only if they have the same skeleton and unshielded colliders. The skeleton and the unshielded colliders are the same within the class. A Markov equivalence class of \glspl{dag} can be represented by a \emph{\gls{cpdag}}, a \gls{pdag} with the same skeleton as the \glspl{dag} in the class and the directed edges that are present in every \gls{dag} in the class. Any \gls{dag} in the class can be obtained by orienting the undirected edges in the \gls{cpdag}, as long as no directed cycles or new unshielded colliders are created. Without further assumptions, structure learning algorithms cannot distinguish between Markov equivalent \glspl{dag}; therefore, \glspl{dag} can be learned up to Markov equivalence.

\begin{dfn}[Hierarchical model]
	Let $\mathbf{X}$ be a set of  variables and $\mathbf{F}$ be a set of inclusion-maximal subsets of $\mathbf{X}$ such that every variable in $\mathbf{X}$ is included in at least one set in $\mathbf{F}$. The \emph{hierarchical model} defined by $\mathbf{F}$ (denoted by $\mathbf{HM}(\mathbf{F})$) is the set of distributions $P$ of $\mathbf{X}$ that are the normalised product of nonnegative functions of the values of the sets in $\mathbf{F}$:
	\begin{displaymath}
		p(\mathbf{x}) = \frac{1}{c} \prod_{\mathbf{A} \in \mathbf{F}} \phi_{\mathbf{A}}(\mathbf{a})
	\end{displaymath}
	where $\phi_{\mathbf{A}}$ is a nonnegative function of the values of $\mathbf{A}$ and
	\begin{displaymath}
		c = \sum_{\mathbf{x}} \prod_{\mathbf{A} \in \mathbf{F}} \phi_{\mathbf{A}}(\hat{\mathbf{a}}) > 0. 
	\end{displaymath}
\end{dfn}

\section{Results}
\label{sec:results}

We start by formally defining conditioning and selection of models.

\begin{dfn}[Model under conditioning]
	Let $\mathbf{M}$ be a model over $\mathbf{O} \dot{\cup} \mathbf{C} \dot{\cup} \mathbf{S}$. Then $\mathbf{M}$ after conditioning on $\mathbf{C}$ and $\mathbf{S} = \hat{\mathbf{s}}$, denoted by $\mathbf{M}[^{\mathbf{C},\mathbf{S} = \hat{\mathbf{s}}}$, is a conditional model over $\mathbf{O}$ given $\mathbf{C}$ defined as follows.
	\begin{align*}
		Q \in \mathbf{M}[^{\mathbf{C},\mathbf{S} = \hat{\mathbf{s}}} \iff \exists P \in \mathbf{M} \text{ s.t. } P[^{\mathbf{C}, \mathbf{S} =\hat{\mathbf{s}}} = Q
	\end{align*}
	When $\mathbf{C} = \emptyset$, $\mathbf{M}$ is said to be under \emph{selection}. The variables in $\mathbf{O}$, $\mathbf{C}$, and $\mathbf{S}$ are referred to as \emph{observed variables}, \emph{conditioning variables}, and \emph{selection variables}, respectively.
\end{dfn}

According the following lemma, model conditioning is commutative.

\begin{lmm}
\label{lmm:mccom}
Let $\mathbf{M}$ be a model over $\mathbf{O} \dot{\cup} \mathbf{C}_1 \dot{\cup} \mathbf{C}_2 \dot{\cup} \mathbf{S}_1 \dot{\cup} \mathbf{S}_2$. Then
\begin{displaymath}
	(\mathbf{M}^{\mathbf{C}_1,\mathbf{S}_1 = \hat{\mathbf{s}}_1})[^{\mathbf{C}_2,\mathbf{S}_2 = \hat{\mathbf{s}}_2} = \mathbf{M}[^{\mathbf{C}_1 \cup \mathbf{C}_2, \mathbf{S}_1 \cup \mathbf{S}_2 = \hat{\mathbf{s}}_1 \cup \hat{\mathbf{s}}_2}.
\end{displaymath}
\end{lmm}
\begin{proof}
The proof follows from commutativity of conditioning and selection of distributions.
\end{proof}

Let $\mathbf{O}$ and $\mathbf{S} = \{S_1, \ldots, S_{n}\}$ be distinct sets of categorical variables, $\mathbf{F} = \{\mathbf{F}_1, \ldots, \mathbf{F}_n\}$ be a set of inclusion-maximal subsets of $\mathbf{O}$ such that every variable in $\mathbf{O}$ is included in at least one set in $\mathbf{F}$, $\mathbf{E}_i$ ($1 \le i \le n$) be the set of edges from each node in $\mathbf{F}_i$ to $S_i$, and $\mathcal{G} = (\mathbf{O} \cup \mathbf{S}, \bigcup_{1 \le i \le n} \mathbf{E}_i)$. \citet{lauritzen1999generating} proved that $\bnmodel(\mathcal{G})[^{\mathbf{S} = \hat{\mathbf{s}}} = \mathbf{HM}(\mathbf{F})$, which means that every categorical hierarchical model is a \gls{bn} model under selection.

\begin{xmp}
\label{xmp:lauritzen_example}
Let $\mathcal{G}$ be the \gls{dag} in Figure \ref{fig:lauritzen_example}. Then 
\begin{displaymath}
	\bnmodel(\mathcal{G})[^{\{S_1, S_2, S_3\} = \{\hat{s_1}, \hat{s_2}, \hat{s_3}\}} = \mathbf{HM}(\{\{O_1, O_2\}, \{O_2, O_3\}, \{O_3, O_1\}\}).
\end{displaymath}
This hierarchical model is called the \emph{no three-way interaction model} in the statistical literature \citep{jobson2012applied}. Note that the corresponding \gls{mag} model is the saturated model; therefore, the constraints in the hierarchical model are non-\gls{ci} constraints.
\end{xmp}

\begin{figure}[htb]
	\centering
	\begin{tikzpicture}
		\node (S1) at (0, 0) {$S_1$};
		\node (O1) at (2, 0) {$O_1$};
		\node (S3) at (4, 0) {$S_3$};
		\node (O2) at (1, -1) {$O_2$};
		\node (O3) at (3, -1) {$O_3$};
		\node (S2) at (2, -2) {$S_2$};
		\draw[->] (O1) -- (S1);
		\draw[->] (O1) -- (S3);
		\draw[->] (O2) -- (S1);
		\draw[->] (O3) -- (S3);
		\draw[->] (O2) -- (S2);
		\draw[->] (O3) -- (S2);
	\end{tikzpicture}
	\caption{The \gls{dag} of a \gls{bn} model which, under conditioning on the values of certain variables ($S_1$, $S_2$, and $S_3$), equals a non-saturated hierarchical model (See Example \ref{xmp:lauritzen_example}).}
	\label{fig:lauritzen_example}
\end{figure}
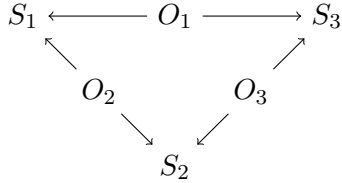

In fact, every \gls{bn} model under selection is a submodel of a (not necessarily saturated) hierarchical model; consider the following example.

\begin{xmp}
\label{xmp:shm}
Let $\mathcal{G}$ be the \gls{dag} in Figure \ref{fig:diamond}. For $P \in \bnmodel(\mathcal{G})$ such that $p(\hat{s}) > 0$ it holds that
\begin{align*}
p(o_1, o_2, o_3 \mid \hat{s}) &= c \cdot p(o_1, o_2, o_3, \hat{s}) = c \cdot p(o_1) p(o_2 \mid o_1) p(o_3 \mid o_1) p(\hat{s} \mid o_2, o_3) \\
&= c \cdot\mathrm{\phi}_{\{O_1, O_2\}}(o_1, o_2)\mathrm{\phi}_{\{O_1,O_3\}}(o_1, o_3)\mathrm{\phi}_{\{O_2,O_3\}}(o_2,o_3)
\end{align*}
where $c = 1/p(\hat{s})$ and $\mathrm{\phi}_{\mathbf{A}}$ is a nonnegative function of the values of $\mathbf{A}$. That is, $P[^{S = \hat{s}} \in \mathbf{HM}(\{\{O_1, O_2\}, \{O_1,O_3\}, \{O_2,O_3\}\})$. Note that this is the same hierarchical model as in Example \ref{fig:lauritzen_example}.
\end{xmp}

\begin{figure}[htb]
	\centering
	\begin{tikzpicture}
		\node (X1) at (1, 0) {$O_1$};
		\node (X2) at (0, -1) {$O_2$};
		\node (X3) at (2, -1) {$O_3$};
		\node (X4) at (1, -2) {$S$};
		\draw[->] (X1) -- (X2);
		\draw[->] (X1) -- (X3);
		\draw[->] (X2) -- (X4);
		\draw[->] (X3) -- (X4);
	\end{tikzpicture}
	\caption{The \gls{dag} of a \gls{bn} model which, under conditioning on the value a certain variable ($S$), is a submodel of a non-saturated hierarchical model (See Example \ref{xmp:shm}).}
	\label{fig:diamond}
\end{figure}

\begin{dfn}[Selection hierarchical model]
Let $\mathbf{O} \dot{\cup} \mathbf{S}$ be a set of variables, $\mathcal{G} = (\mathbf{O} \cup \mathbf{S}, \mathbf{E})$, $\mathbf{I}$ be the set of the intersections of the families in $\mathcal{G}$ with $\mathbf{O}$:
\begin{displaymath}
	\mathbf{I} = \{(X \cup \pain{\mathcal{G}}{X}) \cap \mathbf{O}: X \in \mathbf{O} \cup \mathbf{S}\}
\end{displaymath}
and $\mathbf{F}$ be the inclusion-maximal among the unique sets of $\mathbf{I}$. $\mathbf{HM}(\mathbf{F})$ is called the \emph{selection hierarchical model} of $\bnmodel(\mathcal{G})$ with respect to $\mathbf{S} = \hat{\mathbf{s}}$ (denoted by $\mathbf{SHM}(\mathcal{G}, \mathbf{S} =\hat{\mathbf{s}})$).
\end{dfn}

\begin{lmm}
\label{lmm:shm}
Let $\mathcal{G} = (\mathbf{O} \cup \mathbf{S}, \mathbf{E})$, where $\mathbf{O} \dot{\cup} \mathbf{S}$ is a set of categorical variables. Then
\begin{displaymath}
	\bnmodel(\mathcal{G})[^{\mathbf{S} = \hat{\mathbf{s}}} \subseteq \mathbf{SHM}(\mathcal{G}, \mathbf{S} = \hat{\mathbf{s}}).
\end{displaymath}
\end{lmm}

We refer to the constraints in the selection hierarchical model as \emph{factorisation} constraints. In the end of this section, an example of a non-\gls{ci}, non-factorisation constraint is given.

The results that follow simplify the characterisation problem.

Owing to the lemma below, it is sufficient to characterise \gls{bn} models under selection when the selection nodes are sinks.

\begin{lmm}
\label{lmm:svs_are_sinks}
Let $\mathbf{O} \dot{\cup} \mathbf{S}$ be a set of categorical variables, $\mathcal{G} = (\mathbf{O} \cup \mathbf{S}, \mathbf{E})$, and $\mathcal{G}'$ be the subgraph of $\mathcal{G}$ with all edges out of $\mathbf{S}$ removed. Then
	\begin{displaymath}
		\bnmodel(\mathcal{G}')[^{\mathbf{S} = \hat{\mathbf{s}}} = \bnmodel(\mathcal{G})[^{\mathbf{S} = \hat{\mathbf{s}}}.
	\end{displaymath}
\end{lmm}

Based on the following lemma and due to commutativity of model conditioning (Lemma \ref{lmm:mccom}), only cases where the sets of parents of the selection nodes are non-nested need to be considered.

\begin{lmm}
\label{lmm:svs_parents_are_non_nested}
Let $\mathbf{O} \dot{\cup} \mathbf{S}$ be a set of categorical variables, $\{S_1, S_2\} \subseteq \mathbf{S}$, $\mathcal{G} = (\mathbf{O} \cup \mathbf{S}, \mathbf{E})$ such that $\chin{\mathcal{G}}{\mathbf{S}} = \emptyset$ and $\pain{\mathcal{G}}{S_2} \subseteq \pain{\mathcal{G}}{S_1}$, $\mathbf{S}' = \mathbf{S} \setminus \{S_2\}$, and $\mathcal{G}' = \mathcal{G}_{\mathbf{O} \cup \mathbf{S}'}$. Then
	\begin{displaymath}
		\bnmodel(\mathcal{G}')[^{\mathbf{S}' = \hat{\mathbf{s}}'} = \bnmodel(\mathcal{G})[^{\mathbf{S} = \hat{\mathbf{s}}}.
	\end{displaymath}
\end{lmm}

We note that Lemma \ref{lmm:svs_are_sinks} and \ref{lmm:svs_parents_are_non_nested} are dual to Lemma 1 and 2, respectively, about marginalisation, in \citet{evans2016graphs}. 

The notion of \emph{conditional} \gls{bn} model is needed for the theorem that follows.

\begin{dfn}[Conditional \gls{bn} model]
Let $\mathbf{X}$ and $\mathbf{Y}$ be distinct sets of variables and $\mathcal{G}$ be a conditional \gls{dag} with random nodes $\mathbf{X}$ and fixed nodes $\mathbf{Y}$. The \emph{conditional \gls{bn} model} defined by $\mathcal{G}$, denoted by $\bnmodel(\mathcal{G})$, is the set of conditional distributions $P$ of $\mathbf{X}$ given $\mathbf{Y}$ that satisfy the \emph{conditional Markov condition} with $\mathcal{G}$, that is, every variable $X$ in $\mathbf{X}$ is independent of its non-descendants and non-parents in $\mathcal{G}$ given its parents in $\mathcal{G}$:
	\begin{displaymath}
		\indin{P}{X}{\mathbf{X} \setminus (\dein{\mathcal{G}}{X} \cup \pain{\mathcal{G}}{X})}{\pain{\mathcal{G}}{X}}
	\end{displaymath}
\end{dfn}

Clearly, a \gls{bn} model is a conditional \gls{bn} model.

Let $\mathbf{X}$ and $\mathbf{Y}$ be distinct sets of categorical variables, $\mathcal{G}$ be a conditional \gls{dag} with random nodes $\mathbf{X}$ and fixed nodes $\mathbf{Y}$, and $P \in \bnmodel(\mathcal{G})$. The proof of Theorem 1.4 in \citet{neapolitan2004learning} can be easily adapted to show that $P$ equals the product of the conditional distributions of the random nodes in $\mathcal{G}$ given their parents in $\mathcal{G}$:
\begin{displaymath}
	p(\mathbf{x} \mid \mathbf{y}) = \prod_{X \in \mathbf{X}} p(x \mid \painvalues{\mathcal{G}}{X})
\end{displaymath}

Suppose that conditional distributions of the nodes in $\mathcal{G}$ given their parents in $\mathcal{G}$ are specified. It is straightforward to adapt the proof of Theorem 1.5 in \citet{neapolitan2004learning} to show that the product of the distributions is in $\bnmodel(\mathcal{G})$.

According to the theorem below, selection only affects the ancestors of the selection nodes; it is therefore sufficient to characterise the case where all nodes are ancestors of the selection nodes. 

\begin{thm}
\label{thm:ancestors}
Let $\mathbf{O} \dot{\cup} \mathbf{S}$ be a set of categorical variables, $\mathcal{G} = (\mathbf{O} \cup \mathbf{S}, \mathbf{E})$ such that $\chin{\mathcal{G}}{\mathbf{S}} = \emptyset$, $\mathbf{X} = \anin{\mathcal{G}}{\mathbf{S}}$, $\mathbf{Y} = \mathbf{O} \setminus \mathbf{X}$, $\mathcal{G}_1 = \mathcal{G}_{\mathbf{X}}$, $\mathcal{G}_2 = \phi_{\mathbf{X} \setminus \mathbf{S}}(\mathcal{G}_{\mathbf{O}})$,  $P$ be a distribution over $\mathbf{O}$ such that $p(\mathbf{x} \setminus \mathbf{s}) > 0$, $P_1 = P[_{\mathbf{Y}}$, and $P_2 = P[^{\mathbf{X} \setminus \mathbf{S}}$. Then
	\begin{displaymath}
		P \in \bnmodel(\mathcal{G})[^{\mathbf{S} = \hat{\mathbf{s}}} \iff P_1 \in \bnmodel(\mathcal{G}_1)[^{\mathbf{S} = \hat{\mathbf{s}}} \land P_2 \in \bnmodel(\mathcal{G}_2)
	\end{displaymath}
\end{thm}

\begin{xmp}
\label{xmp:ancestors}
	Let $\mathbf{O} = \{O_1, \ldots, O_6\}$ be a set of categorical variables, $S$ be a variable not in $\mathbf{O}$, $\mathcal{G}$, $\mathcal{G}_1$, and $\mathcal{G}_2$ be the \gls{dag} in Figure \ref{fig:ancestor_theorem_example_dag_1}, \ref{fig:ancestor_theorem_example_dag_2}, and \ref{fig:ancestor_theorem_example_dag_3}, respectively, $P$ be a distribution over $\mathbf{O}$ such that $p(o_1, o_2, o_3, o_4) > 0$, $P_1 = P[_{\{O_5, O_6\}}$, and $P_2 = P[^{\{O_1, O_2, O_3, O_4\}}$. Then Theorem \ref{thm:ancestors} says that
\begin{displaymath}
	P \in \bnmodel(\mathcal{G})[^{S = \hat{s}} \iff P_1 \in \bnmodel(\mathcal{G}_1)[^{S = \hat{s}} \land P_2 \in \bnmodel(\mathcal{G}_2).
\end{displaymath}
Therefore, only $\bnmodel(\mathcal{G}_1)[^{S = \hat{s}}$ needs to be characterised.
\end{xmp}

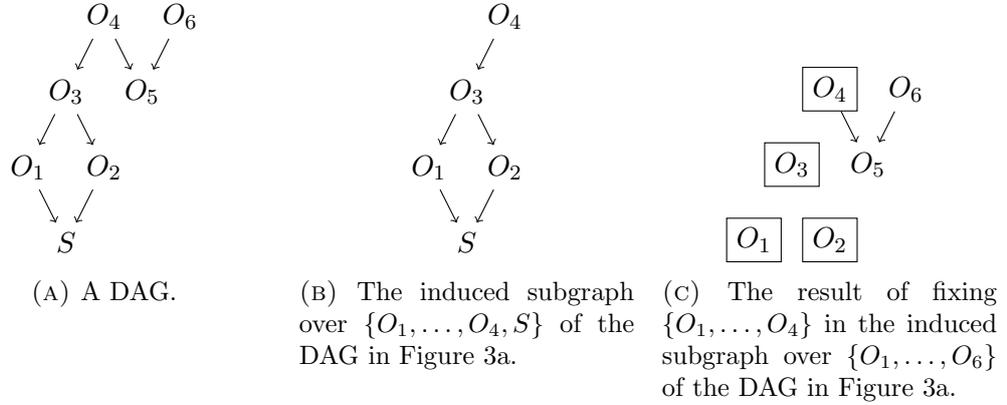
\begin{figure}[htb]
\centering
        \begin{subfigure}[t]{0.3\columnwidth}
		\centering
		\begin{tikzpicture}
			\node (O3) at (0.5, -1) {$O_3$};
			\node (O4) at (1, 0) {$O_4$};
			\node (O6) at (2, 0) {$O_6$};
			\node (O1) at (0, -2) {$O_1$};
			\node (O2) at (1, -2) {$O_2$};
			\node (O5) at (1.5, -1) {$O_5$};
			\node (S) at (0.5, -3) {$S$};
			\draw[->] (O3) -- (O1);
			\draw[->] (O3) -- (O2);
			\draw[->] (O4) -- (O3);
			\draw[->] (O4) -- (O5);
			\draw[->] (O6) -- (O5);
			\draw[->] (O1) -- (S);
			\draw[->] (O2) -- (S);
		\end{tikzpicture}
		\caption{A \gls{dag}.}
		\label{fig:ancestor_theorem_example_dag_1}
        \end{subfigure}
        \;
        \begin{subfigure}[t]{0.3\columnwidth}
		\centering
		\begin{tikzpicture}
			\node (O3) at (0.5, -1) {$O_3$};
			\node (O4) at (1, 0) {$O_4$};
			\node (O1) at (0, -2) {$O_1$};
			\node (O2) at (1, -2) {$O_2$};
			\node (S) at (0.5, -3) {$S$};
			\draw[->] (O3) -- (O1);
			\draw[->] (O3) -- (O2);
			\draw[->] (O4) -- (O3);
			\draw[->] (O1) -- (S);
			\draw[->] (O2) -- (S);
		\end{tikzpicture}
		\caption{The induced subgraph over $\{O_1, \ldots, O_4, S\}$ of the \gls{dag} in Figure \ref{fig:ancestor_theorem_example_dag_1}.}
		\label{fig:ancestor_theorem_example_dag_2}
        \end{subfigure}
        \;
        \begin{subfigure}[t]{0.3\columnwidth}
		\centering
		\begin{tikzpicture}
			\node (O3) at (0.5, -1) [rectangle, draw = black] {$O_3$};
			\node (O4) at (1, 0) [rectangle, draw = black] {$O_4$};
			\node (O6) at (2, 0) {$O_6$};
			\node (O1) at (0, -2) [rectangle, draw = black] {$O_1$};
			\node (O2) at (1, -2) [rectangle, draw = black] {$O_2$};
			\node (O5) at (1.5, -1) {$O_5$};
			\draw[->] (O4) -- (O5);
			\draw[->] (O6) -- (O5);
		\end{tikzpicture}
		\caption{The result of fixing $\{O_1, \ldots, O_4\}$ in the induced subgraph over $\{O_1, \ldots, O_6\}$ of the \gls{dag} in Figure \ref{fig:ancestor_theorem_example_dag_1}.}
		\label{fig:ancestor_theorem_example_dag_3}
        \end{subfigure}
	\caption{The graphs in Example \ref{xmp:ancestors}.}
\end{figure}

In general, different Markov equivalent \glspl{dag} have different ancestors of the selection nodes. When computing constraints in a model under selection, Theorem \ref{thm:ancestors} may be applied to a Markov equivalent \gls{dag} with the fewest ancestors of the selection nodes. As it turns out, such a \gls{dag} contains only the \emph{compelled ancestors} of the selection nodes. Let $\mathbf{G}$ be a Markov equivalence class of \glspl{dag}. $X$ is a \emph{compelled ancestor} of $Y$ in $\mathbf{G}$ if $X \in \anin{\mathcal{G}}{Y}$ for every $\mathcal{G} \in \mathbf{G}$. Clearly, compelled ancestry is a transitive relation. The union of the sets of compelled ancestors of $\mathbf{Y}$ in $\mathbf{G}$ is denoted by $\companin{\mathbf{G}}{\mathbf{Y}}$. Note that, if $\mathbf{G}$ is the Markov equivalence class of an (unknown) causal \gls{dag}, then $\companin{\mathbf{G}}{\mathbf{Y}}$ is the set of \emph{definite causes} of $Y$.

The theorem that follows states that, for a Markov equivalence class of \glspl{dag}, there exists a \gls{dag} in the class such that the ancestors of a set of variables in the \gls{dag} are the compelled ancestors of the variables in the class.

\begin{thm}
\label{thm:compelled_ancestor_dag_edges}
Let $\mathbf{G}$ be a Markov equivalence class of \glspl{dag} over $\mathbf{V}$ and $\mathbf{Y} \subseteq \mathbf{V}$. There exists $\mathcal{G} \in \mathbf{G}$ such that $\anin{\mathcal{G}}{\mathbf{Y}} = \companin{\mathbf{G}}{\mathbf{Y}}$.
\end{thm}

The lemma below is useful to obtaining such a \gls{dag}.

\begin{lmm}
\label{lmm:compelled_ancestor_dag_condition}
Let $\mathbf{G}$ be a Markov equivalence class of \glspl{dag} over $\mathbf{V}$ and $\mathcal{G} \in \mathbf{G}$. Let further $\mathbf{Y} \subseteq \mathbf{V}$ and $\mathbf{Z} = \companin{\mathbf{G}}{\mathbf{Y}}$. $\anin{\mathcal{G}}{\mathbf{Y}} = \mathbf{Z}$ if and only if every edge between $Z \in \mathbf{Z}$ and $X \in \mathbf{V} \setminus \mathbf{Z}$ in $\mathcal{G}$ is out of $Z$.
\end{lmm}

The following theorem gives necessary and sufficient criteria for identifying compelled ancestors from a \gls{cpdag}, thereby eliminating the need to enumerate the \glspl{dag} in the Markov equivalence class and take the intersection of the ancestors in each \gls{dag}.

\begin{thm}
\label{thm:compelled_ancestor}
Let $\mathbf{G}$ be a Markov equivalence class of \glspl{dag} and $\mathcal{P}$ be its \gls{cpdag}. Then $X \in \companin{\mathbf{G}}{\mathbf{Y}}$ if and only if either
\begin{enumerate}
	\item $X \in \anin{\mathcal{P}}{\mathbf{Y}}$, or
	\item $X$ is on an unshielded undirected path between two members of $\anin{\mathcal{P}}{\mathbf{Y}}$ in $\mathcal{P}$.
\end{enumerate}
\end{thm}

Algorithm \ref{alg:find_compelled_ancestors} can be used for the identification of compelled ancestors based on Theorem \ref{thm:compelled_ancestor}. The algorithm first identifies $\mathbf{A} = \anin{\mathcal{P}}{\mathbf{Y}}$ and then recursively finds the interior nodes $\mathbf{U}$ on the unshielded undirected paths from each member $X$ of $\mathbf{A}$ to $\mathbf{A} \setminus \{X\}$ with no interior node in $\mathbf{A}$ in $\mathcal{P}$. The second part of the algorithm is based on the lemma below.

\begin{lmm}
\label{lmm:u_up_subgraph_is_tree}
Let $\mathcal{P}$ be a \gls{cpdag} over $\mathbf{V}$ and $X \in \mathbf{V}$. The subgraph of $\mathcal{P}$ containing only the nodes and the edges on unshielded undirected paths from $X$ to $\mathbf{V} \setminus \{X\}$ in $\mathcal{P}$ is a tree.
\end{lmm}

For each member $X$ of $\mathbf{A}$, function call $\textsc{find\_uup\_interior\_nodes}(\mathcal{U}, X, \mathbf{A})$ at line \ref{line:find_uup_interior_nodes_call} of Algorithm \ref{alg:find_compelled_ancestors} implicitly performs a \emph{\gls{dfs}} in the tree of unshielded undirected paths from $X$ to $\mathbf{V} \setminus \{X\}$ in $\mathcal{P}$, starting at $X$. When either a member of $\mathbf{A}$ or a leaf in the tree is encountered, the algorithm backtracks. When at node $Y$ during backtracking, if a member of $\mathbf{A}$ was encountered beyond $Y$ (as indicated by the boolean flag $found$), $Y$ is added to $\mathbf{U}$ (unless $Y \in \mathbf{A}$). Since identifying $\mathbf{A}$ is $O(|\mathbf{V}|)$, each \gls{dfs} is $O(|\mathbf{V}|)$, and a \gls{dfs} is performed for each member of $\mathbf{A}$, Algorithm \ref{alg:find_compelled_ancestors} is $O(|\mathbf{V}|^2)$.

\begin{algorithm}
\caption{Find compelled ancestors. $\mathcal{P}$ is a \gls{cpdag} over $\mathbf{V}$ and $\mathbf{Y} \subseteq \mathbf{V}$. In the output, $\mathbf{CA} = \companin{\mathbf{G}}{\mathbf{Y}}$, where $\mathbf{G}$ is the Markov equivalence class of \glspl{dag} represented by $\mathcal{P}$. The \textsc{uup} in function names stands for unshielded undirected path.}
\label{alg:find_compelled_ancestors}
\begin{algorithmic}[1]
\Require $\mathcal{P}, \mathbf{Y}$
\Ensure $\mathbf{CA}$
\State{$\mathbf{A} \gets \anin{\mathcal{P}}{\mathbf{Y}}$}
\State{Let $\mathcal{U}$ be the subgraph of $\mathcal{P}$ containing only the undirected edges of $\mathcal{P}$}
\State{$\mathbf{U} \gets \emptyset$}
\ForAll{$X \in \mathbf{A}$}
	\State{$\mathbf{U}_X \gets \Call{find\_uup\_interior\_nodes}{\mathcal{U}, X, \mathbf{A}}$} \label{line:find_uup_interior_nodes_call}
	\State{$\mathbf{U} \gets \mathbf{U} \cup \mathbf{U}_X$}
\EndFor
\State{$\mathbf{CA} \gets \mathbf{A} \cup \mathbf{U}$}
\Function{find\_uup\_interior\_nodes}{$\mathcal{U}, X, \mathbf{A}$}
	\State{$\mathbf{U} \gets \emptyset$}
	\ForAll{$Y \in \adjin{\mathcal{U}}{X}$}
		\State{$(found_Y, \mathbf{U}_Y)  \gets \Call{find\_uup\_interior\_nodes\_recursive}{\mathcal{U}, X, Y, \mathbf{A}}$}
		\State{$\mathbf{U} \gets \mathbf{U} \cup \mathbf{U}_Y$}
	\EndFor
	\State{\Return $\mathbf{U}$}
\EndFunction
\Function{find\_uup\_interior\_nodes\_recursive}{$\mathcal{U},  X, Y, \mathbf{A}$}
	\State{$found \gets false$}
	\State{$\mathbf{U} \gets \emptyset$}
	\ForAll{$Z \in \adjin{\mathcal{U}}{Y} \setminus (\adjin{\mathcal{U}}{X} \cup \{X\})$}
		\If{$Z \in \mathbf{A}$}
			\State{$found_Z \gets true$}
			\State{$\mathbf{U}_Z \gets \emptyset$}
		\Else
			\State{$(found_Z, \mathbf{U}_Z) \gets \Call{find\_uup\_interior\_nodes\_recursive}{\mathcal{U}, Y, Z, \mathbf{A}}$}
		\EndIf
		\State{$\mathbf{U} \gets \mathbf{U} \cup \mathbf{U}_Z$}
		\State{$found \gets found \lor found_Z$}
	\EndFor
	\If{$found$}
		\State{$\mathbf{U} \gets \mathbf{U} \cup \{Y\}$}
	\EndIf
	\State{\Return ($found, \mathbf{U})$}
\EndFunction
\end{algorithmic}
\end{algorithm}

Let $\mathcal{G}$ be a \gls{dag} over $\mathbf{V}$, $\mathbf{G}$ be the Markov equivalence class of $\mathcal{G}$, $\mathbf{Y} \subseteq \mathbf{V}$, and $\mathbf{Z} = \companin{\mathbf{G}}{\mathbf{Y}}$. A \gls{dag} $\mathcal{G}'$ in $\mathbf{G}$ such that $\anin{\mathcal{G}'}{\mathbf{Y}} = \mathbf{Z}$ can be obtained as follows. First, obtain the \gls{cpdag} $\mathcal{P}$ of $\mathbf{G}$ \citep[using, e.g., the algorithm of][]{chickering1995transformational}. Then identify $\mathbf{Z}$ from $\mathcal{P}$ using Algorithm \ref{alg:find_compelled_ancestors}. Finally, use Algorithm 10.3 in \citet{neapolitan2004learning}, orienting edges between $Z \in \mathbf{Z}$ and $X \in \mathbf{V} \setminus \mathbf{Z}$ out of $Z$ in $\mathcal{P}$ in Step 1 of the algorithm, to orient edges in $\mathcal{P}$.

\begin{xmp}
\label{xmp:fewest_ancestors}
Let $G$ be the \gls{dag} in Figure \ref{fig:ancestor_theorem_example_dag_1} and $\mathbf{G}$ be the Markov equivalence class where $G$ belongs. Figure \ref{fig:ancestor_colliders_corollary_example_cpdag} shows the \gls{cpdag} of $G$. According to Theorem \ref{thm:compelled_ancestor},
\begin{displaymath}
	\companin{\mathbf{G}}{S} = \{O_1, O_2, O_3\}.
\end{displaymath}
Owing to Theorem \ref{thm:compelled_ancestor_dag_edges}, there exists $\mathcal{G}' \in \mathbf{G}$ such that $\anin{\mathcal{G}'}{S} = \{O_1, O_2, O_3\}$. Such $G'$ is the \gls{dag} in Figure \ref{fig:ancestor_colliders_corollary_example_dag_1}. Let $\mathcal{G}'_1$ and $\mathcal{G}'_2$ be the \gls{dag} in Figures \ref{fig:ancestor_colliders_corollary_example_dag_2} and \ref{fig:ancestor_colliders_corollary_example_dag_3}, respectively, $P$ be a distribution over $\mathbf{O}$ such that $p(o_1, o_2, o_3) > 0$, $P'_1 = P[_{\{O_4, O_4, O_5, O_6\}}$, and $P'_2 = P[^{\{O_1, O_2, O_3\}}$. Then Theorem \ref{thm:ancestors} implies that
\begin{displaymath}
	P \in \bnmodel(\mathcal{G})[^{S = \hat{s}} = \bnmodel(\mathcal{G}')[^{S = \hat{s}}\iff P_1 \in \bnmodel(\mathcal{G}_1')[^{S = \hat{s}} \land P_2 \in \bnmodel(\mathcal{G}_2'). 
\end{displaymath}
Therefore, only $\bnmodel(\mathcal{G}_1')[^{S = \hat{s}}$, which has fewer variables than $\bnmodel(\mathcal{G}_1)[^{S = \hat{s}}$ in Example \ref{xmp:ancestors}, needs to be characterised.
\end{xmp}

\begin{figure}[htb]
\centering
        \begin{subfigure}[t]{0.45\columnwidth}
		\centering
		\begin{tikzpicture}
			\node (O3) at (0.5, -1) {$O_3$};
			\node (O4) at (1, 0) {$O_4$};
			\node (O6) at (2, 0) {$O_6$};
			\node (O1) at (0, -2) {$O_1$};
			\node (O2) at (1, -2) {$O_2$};
			\node (O5) at (1.5, -1) {$O_5$};
			\node (S) at (0.5, -3) {$S$};
			\draw[-] (O3) -- (O1);
			\draw[-] (O3) -- (O2);
			\draw[-] (O4) -- (O3);
			\draw[->] (O4) -- (O5);
			\draw[->] (O6) -- (O5);
			\draw[->] (O1) -- (S);
			\draw[->] (O2) -- (S);
		\end{tikzpicture}
		\caption{The \gls{cpdag} of the Markov equivalence class where the \gls{dag} in Figure \ref{fig:ancestor_theorem_example_dag_1} belongs. $O_1$, $O_2$, and $O_3$ are the compelled ancestors of $S$ in the class.}
		\label{fig:ancestor_colliders_corollary_example_cpdag}
        \end{subfigure}
        \quad
        \begin{subfigure}[t]{0.45\columnwidth}
		\centering
		\begin{tikzpicture}
			\node (O3) at (0.5, -1) {$O_3$};
			\node (O4) at (1, 0) {$O_4$};
			\node (O6) at (2, 0) {$O_6$};
			\node (O1) at (0, -2) {$O_1$};
			\node (O2) at (1, -2) {$O_2$};
			\node (O5) at (1.5, -1) {$O_5$};
			\node (S) at (0.5, -3) {$S$};
			\draw[->] (O3) -- (O1);
			\draw[->] (O3) -- (O2);
			\draw[->] (O3) -- (O4);
			\draw[->] (O4) -- (O5);
			\draw[->] (O6) -- (O5);
			\draw[->] (O1) -- (S);
			\draw[->] (O2) -- (S);
		\end{tikzpicture}
		\caption{A \gls{dag} in the Markov equivalence class represented by the \gls{cpdag} in Figure \ref{fig:ancestor_colliders_corollary_example_cpdag} such that the only ancestors of $S$ are the compelled ancestors of $S$ in the class.}
		\label{fig:ancestor_colliders_corollary_example_dag_1}
        \end{subfigure}
        \par \bigskip
        \begin{subfigure}[t]{0.45\columnwidth}
		\centering
		\begin{tikzpicture}
			\node (O3) at (0.5, -1) {$O_3$};
			\node (O1) at (0, -2) {$O_1$};
			\node (O2) at (1, -2) {$O_2$};
			\node (S) at (0.5, -3) {$S$};
			\draw[->] (O3) -- (O1);
			\draw[->] (O3) -- (O2);
			\draw[->] (O1) -- (S);
			\draw[->] (O2) -- (S);
		\end{tikzpicture}
		\caption{The induced subgraph over $\{O_1, O_2, O_3, S\}$ of the \gls{dag} in Figure \ref{fig:ancestor_colliders_corollary_example_dag_1}.}
		\label{fig:ancestor_colliders_corollary_example_dag_2}
        \end{subfigure}
        \quad
        \begin{subfigure}[t]{0.45\columnwidth}
		\centering
		\begin{tikzpicture}
			\node (O3) at (0.5, -1) [rectangle, draw = black] {$O_3$};
			\node (O4) at (1, 0) {$O_4$};
			\node (O6) at (2, 0) {$O_6$};
			\node (O1) at (0, -2) [rectangle, draw = black] {$O_1$};
			\node (O2) at (1, -2) [rectangle, draw = black] {$O_2$};
			\node (O5) at (1.5, -1) {$O_5$};
			\draw[->] (O3) -- (O4);
			\draw[->] (O4) -- (O5);
			\draw[->] (O6) -- (O5);
		\end{tikzpicture}
		\caption{The result of fixing $\{O_1, O_2, O_3)\}$ in the induced subgraph over $\{O_1, \ldots, O_6\}$ of the \gls{dag} in Figure \ref{fig:ancestor_colliders_corollary_example_dag_1}.}
		\label{fig:ancestor_colliders_corollary_example_dag_3}
        \end{subfigure}
	\caption{The graphs in Example \ref{xmp:fewest_ancestors}.}
\end{figure}
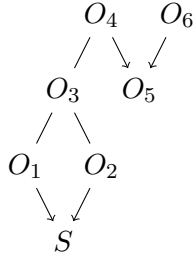
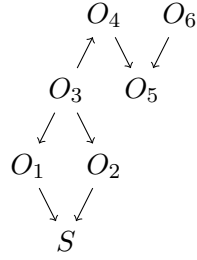
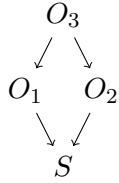
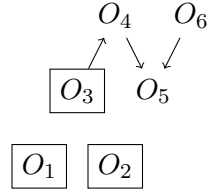

Finally, the following theorem states that, in the case of a single selection node, selection reduces to conditioning on the parents of the node in the induced subgraph over the observed nodes. 

\begin{thm}
\label{thm:selection_conditioning}
Let $\mathbf{O} \dot{\cup} S$ be a set of categorical variables, $\mathcal{G} = (\mathbf{O} \cup S, \mathbf{E})$ such that $\chin{\mathcal{G}}{S} = \emptyset$, and $P$ be a distribution over $\mathbf{O}$ such that $p(\painvalues{\mathcal{G}}{S}) > 0$. Then
	\begin{displaymath}
		P \in \bnmodel(\mathcal{G})[^{S = \hat{s}} \iff P[^{\pain{\mathcal{G}}{S}} \in \bnmodel(\mathcal{G}_{\mathbf{O}})[^{\pain{\mathcal{G}}{S}}
	\end{displaymath}
\end{thm}

We conclude this section with an example of a non-\gls{ci}, non-factorisation constraint in a categorical \gls{bn} model under selection.

\begin{xmp}
\label{xmp:constraint}
Let $\mathcal{G}$ be the \gls{dag} in Figure \ref{fig:case_3} and $\mathbf{O} = \{O_1, O_2, O_3, O_4\}$ be a set of variables with domain $\{1, 2\}$. Owing to Theorem \ref{thm:selection_conditioning}, $P \in \bnmodel(\mathcal{G})[^{S = \hat{s}} \iff P[^{O_4} \in \bnmodel(\mathcal{G}_{\mathbf{O}})[^{O_4}$. Therefore, only $\bnmodel(\mathcal{G}_{\mathbf{O}})[^{O_4}$ needs to be characterised. $\bnmodel(\mathcal{G}_{\mathbf{O}})$ is the set of distributions of $\mathbf{O}$ for which $\ind{O_1}{O_3}{O_2}$ holds. Thus, apart from being nonnegative and sum to one, the values of $p$ must satisfy the following system of polynomial equations:
\scriptsize
\begin{align}
&(p_{1,1,1,1} + p_{1,1,1,2})(p_{2,1,2,1} - p_{2,1,2,2}) - (p_{1,1,2,1} + p_{1,1,2,2})(p_{2,1,1,1} + p_{2,1,1,2}) = 0
\end{align}
\begin{align}
&(p_{1,2,1,1} + p_{1,2,1,2})(p_{2,2,2,1} - p_{2,2,2,2}) - (p_{1,2,2,1} + p_{1,2,2,2})(p_{2,2,1,1} + p_{2,2,1,2}) = 0
\end{align}
\normalsize
where $p_{i,j,k,l}$ stands for $p(O_1 = i, O_2 = j, O_3 = k, O_4 = l)$. Suppose that $Q$ is a distribution over $\mathbf{O}$. In order for $Q$ to be in $\bnmodel(\mathcal{G}_{\mathbf{O}})[^{O_4}$, there must exist a distribution $R$ of $O_4$ such that the $Q \cdot R \in \bnmodel(\mathcal{G}_{\mathbf{O}})$. Let  $q_{i,j,k,l}$ stand for $q(O_1 = i, O_2 = j, O_3 = k \mid O_4 = l)$ and $r_i$ stand for $r(O_4 = i)$. Replacing $p_{i,j,k,l}$ with $q_{i,j,k,l}r_l$ in the system above and adding the equation $r_1 + r_2 = 1$ results in the following system:
\scriptsize
\begin{align}
&(q_{1,1,1,1}r_1 + q_{1,1,1,2}r_2) (q_{2,1,2,1}r_1 - q_{2,1,2,2}r_2) - (q_{1,1,2,1}r_1 + q_{1,1,2,2}r_2) (q_{2,1,1,1}r_1 + q_{2,1,1,2}r_2) = 0
\end{align}
\begin{align}
&(q_{1,2,1,1}r_1 + q_{1,2,1,2}r_2) (q_{2,2,2,1}r_1 - q_{2,2,2,2}r_2) - (q_{1,2,2,1}r_1 + q_{1,2,2,2}r_2) (q_{2,2,1,1}r_1 + q_{2,2,1,2}r_2) = 0
\end{align}
\begin{align}
r_1 + r_2 - 1 = 0 \label{eq:sum_constraint}
\end{align}
\normalsize
The values of $q$ must be such that the system, considered as a system of $r_1$ and $r_2$, has a solution. Replacing 1 with the dummy variable $h$ in Equation \ref{eq:sum_constraint} results in a system of \emph{homogeneous} polynomial equations in $r_1$, $r_2$, and $h$, and the system must have a \emph{nontrivial} solution (that is, a solution other than all variables being zero). For a system of $n$ homogeneous polynomials in $n$ variables, the \emph{resultant} is the unique (up to a constant) polynomial in the coefficients of the polynomials (here, the values of $q$) whose vanishing is equivalent to the system having trivial solutions \citep{cox2006using}, and is irreducible. We used Macaulay 2 to compute the resultant, which has degree 8 and 218 terms. As both the \gls{mag} model over $\mathbf{O}$ and $\mathbf{SHM}($G$, S = \hat{s})$ are saturated, the resultant is an example of a non-\gls{ci}, non-factorisation constraint.

While our concern in this example is the existence of solutions to the system, \citet{evans2015recovering} derived a necessary and sufficient condition for the existence of \emph{finitely many} solutions when $Q \in \bnmodel(\mathcal{G}_{\mathbf{O}})[^{O_4}$ (clearly, a solution exists in this case). Their goal was to recover a distribution $P \in \bnmodel(\mathcal{G})$ from its conditional $Q$ given $S$. If the number of solutions to the system is finite, then the system can be solved (in principle) in order to recover $P$, as $r_1 = p_1$ and $r_2 = p_2$, where $p_i = p(O_1 = i)$, will be one of the solutions, and $p_{i,j,k,l} = q_{i,j,k,l}p_l$.
\end{xmp}

\begin{figure}[htb]
	\centering
	\begin{tikzpicture}
		\node (O1) at (0, 0) {$O_1$};
		\node (O2) at (1, 0) {$O_2$};
		\node (O3) at (2, 0) {$O_3$};
		\node (O4) at (1, -1) {$O_4$};
		\node (S) at (1, -2) {$S$};
		\draw[->] (O1) -- (O2);
		\draw[->] (O1) -- (O4);
		\draw[->] (O2) -- (O3);
		\draw[->] (O2) -- (O4);
		\draw[->] (O3) -- (O4);
		\draw[->] (O4) -- (S);
	\end{tikzpicture}
	\caption{The \gls{dag} of a \gls{bn} model which, under conditioning on the value a certain variable ($S$), includes a non-\gls{ci}, non-factorisation constraint (See Example \ref{xmp:constraint}).}
	\label{fig:case_3}
\end{figure}
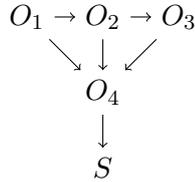

\section{Conclusion and future work}

In this work, preliminary results were provided towards characterising the constraints imposed on a distribution in a \gls{bn} model under selection. Specifically, it was shown that the only cases that need to be characterised are the ones where the selection nodes have no children, the sets of parents of the selection nodes are non-nested, and all nodes are compelled ancestors of the selection nodes. In addition, in the case of a single selection node, selection reduces to conditioning on the parents of the selection node in the induced subgraph over the observed nodes. Furthermore, an algorithm was designed for identifying compelled ancestors from the \gls{cpdag}, thereby eliminating the need to enumerate the \glspl{dag} in the Markov equivalence class. This is a useful result on its own in causal structure learning, as compelled ancestors correspond to definite causes. Finally, a non-\gls{ci}, non-factorisation constraint in a \gls{bn} model under selection was computed for the first time.

Future work includes further reducing the characterisation problem, interpreting the constraints, devising a graph representation of them, unifying that representation with \glspl{mdag}, characterising the equivalence classes of the unified graphs, and, ultimately, devising structure-learning algorithms for the graphs.

\newpage

\appendix
\section*{Appendix -- proofs}
\label{sec:appendix}

\begin{proof}[Proof of Lemma \ref{lmm:shm}]
The proof follows from the factorisation imposed by $\bnmodel(\mathcal{G})$ on its distributions.
\end{proof}

\begin{proof}[Proof of Lemma \ref{lmm:svs_are_sinks}]
Since $\mathcal{G}' \subseteq \mathcal{G}$, $\bnmodel(\mathcal{G}') \subseteq \bnmodel(\mathcal{G})$. Thus, $\bnmodel(\mathcal{G}')[^{\mathbf{S} = \hat{\mathbf{s}}} \subseteq \bnmodel(\mathcal{G})[^{\mathbf{S} = \hat{\mathbf{s}}}$. Now suppose $P \in \bnmodel(\mathcal{G})[^{\mathbf{S} = \hat{\mathbf{s}}}$. Then there exists $Q \in \bnmodel(\mathcal{G})$ such that $Q[^{\mathbf{S} = \hat{\mathbf{s}}} = P$. Let $R \in \bnmodel(\mathcal{G}')$ such that $r(x \mid \painvalues{\mathcal{G}'}{X}) = q(x \mid \painvalues{\mathcal{G}'}{X}, \painfixedvalues{\mathcal{G}}{X} \cap \hat{\mathbf{s}})$ for each $X \in \mathbf{O} \cup \mathbf{S}$. Then
\begin{align*}
r(\mathbf{o}, \hat{\mathbf{s}}) &= \prod_{O \in \mathbf{O}} r(o \mid \painvalues{\mathcal{G}'}{O}) \prod_{S \in \mathbf{S}} r(\hat{s} \mid \painvalues{\mathcal{G}'}{S}) \\
&= \prod_{O \in \mathbf{O}} q(o \mid \painvalues{\mathcal{G}'}{O}, \painfixedvalues{\mathcal{G}}{O} \cap \hat{\mathbf{s}}) \prod_{S \in \mathbf{S}} q(\hat{s} \mid \painvalues{\mathcal{G}'}{S}, \painfixedvalues{\mathcal{G}}{S} \cap \hat{\mathbf{s}}) = q(\mathbf{o}, \hat{\mathbf{s}})
\end{align*}
Therefore, $r(\mathbf{o} \mid \hat{\mathbf{s}}) = q(\mathbf{o} \mid \hat{\mathbf{s}}) = p(\mathbf{o})$. Thus, $P \in \bnmodel(\mathcal{G}')[^{\mathbf{S} = \hat{\mathbf{s}}}$.
\end{proof}

\begin{proofln}[Proof of Lemma \ref{lmm:svs_parents_are_non_nested}]
\emph{Forward direction:} Suppose $P \in \bnmodel(\mathcal{G}')[^{\mathbf{S}' = \hat{\mathbf{s}}'}$. Then there exists $Q \in \bnmodel(\mathcal{G}')$ such that $Q[^{\mathbf{S}' = \hat{\mathbf{s}}'} = P$. Let $R \in \bnmodel(\mathcal{G})$ such that $r(z \mid \painvalues{\mathcal{G}}{Z}) = q(z \mid \painvalues{\mathcal{G}'}{Z})$ for each $Z \in \mathbf{O} \cup \mathbf{S}'$ and $r(\hat{s}_2 \mid \painvalues{\mathcal{G}}{S_2}) = 1$. Then
\begin{align*}
r(\mathbf{o}, \hat{\mathbf{s}}) &= r(\hat{s}_2 \mid \painvalues{\mathcal{G}}{S_2}) \prod_{O \in \mathbf{O}} r(o \mid \painvalues{\mathcal{G}}{O}) \prod_{S' \in \mathbf{S}'} r(\hat{s}' \mid \painvalues{\mathcal{G}}{S'}) \\
&= \prod_{O \in \mathbf{O}} q(o \mid \painvalues{\mathcal{G}}{O}) \prod_{S' \in \mathbf{S}'} q(\hat{s}' \mid \painvalues{\mathcal{G}}{S'}) = q(\mathbf{o}, \hat{\mathbf{s}}')
\end{align*}
Therefore, $r(\mathbf{o} \mid \hat{\mathbf{s}}) = q(\mathbf{o} \mid \hat{\mathbf{s}}') = p(\mathbf{o})$. Thus, $P \in \bnmodel(\mathcal{G})[^{\mathbf{S} = \hat{\mathbf{s}}}$.

\emph{Reverse direction:} Suppose $P \in \bnmodel(\mathcal{G})[^{\mathbf{S} = \hat{\mathbf{s}}}$. Then there exists $Q \in \bnmodel(\mathcal{G})$ such that $Q[^{\mathbf{S} = \hat{\mathbf{s}}} = P$. Let $\mathbf{Z} = \mathbf{S}' \setminus \{S_1\}$ and $R \in \bnmodel(\mathcal{G}')$ such that $r(u \mid \painvalues{\mathcal{G}'}{U}) = q(u \mid \painvalues{\mathcal{G}}{U})$ for each $U \in \mathbf{O} \cup \mathbf{Z}$ and $r(\hat{s}_1 \mid \painvalues{\mathcal{G}'}{S_1}) = q(\hat{s}_1 \mid \painvalues{\mathcal{G}}{S_1}) \cdot q(\hat{s}_2 \mid \painvalues{\mathcal{G}}{S_2})$. Then
\begin{align*}
r(\mathbf{o}, \hat{\mathbf{s}}') &= r(\hat{s}_1 \mid \painvalues{\mathcal{G}'}{S_1}) \prod_{O \in \mathbf{O}} r(o \mid \painvalues{\mathcal{G}'}{O}) \prod_{Z \in \mathbf{Z}} r(\hat{z} \mid \painvalues{\mathcal{G}'}{Z}) \\
&= q(\hat{s}_1 \mid \painvalues{\mathcal{G}}{S_1}) \cdot q(\hat{s_2} \mid \painvalues{\mathcal{G}}{S_2}) \prod_{O \in \mathbf{O}} q(o \mid \painvalues{\mathcal{G}}{O}) \prod_{Z \in \mathbf{Z}} q(\hat{z} \mid \painvalues{\mathcal{G}}{Z}) \\
&= \prod_{O \in \mathbf{O}} q(o \mid \painvalues{\mathcal{G}}{O}) \prod_{S \in \mathbf{S}} q(\hat{s} \mid \painvalues{\mathcal{G}}{S}) = q(\mathbf{o}, \hat{\mathbf{s}})
\end{align*}
Therefore, $r(\mathbf{o} \mid \hat{\mathbf{s}}') = q(\mathbf{o} \mid \hat{\mathbf{s}}) = p(\mathbf{o})$. Thus, $P \in \bnmodel(\mathcal{G}')[^{\mathbf{S} = \hat{\mathbf{s}}'}$.
\end{proofln}

\begin{proofln}[Proof of Theorem \ref{thm:ancestors}]
\emph{Forward direction:} Suppose $P \in \bnmodel(\mathcal{G})[^{\mathbf{S} = \hat{\mathbf{s}}}$. Then there exists $Q \in \bnmodel(\mathcal{G})$ such that $Q[^{\mathbf{S} = \hat{\mathbf{s}}} = P$. Since $p(\mathbf{x} \setminus \mathbf{s}) > 0$, $q(\mathbf{x} \setminus \mathbf{s}, \hat{\mathbf{s}}) > 0$. Let $R_1 \in \bnmodel(\mathcal{G}_1)$ such that $r_1(x \mid \painvalues{\mathcal{G}_1}{X}) = q(x \mid \painvalues{\mathcal{G}}{X})$ for each $X \in \mathbf{X}$. Then
\begin{align*}
q(\mathbf{x}) &= \sum_{\mathbf{y}} q(\mathbf{x}, \mathbf{y}) = \prod\limits_{X \in \mathbf{X}} q(x \mid \painvalues{\mathcal{G}}{X}) \sum\limits_{\mathbf{y}} \prod\limits_{Y \in \mathbf{Y}} q(y \mid \painvalues{\mathcal{G}}{Y} \cap \mathbf{x}, \painvalues{\mathcal{G}}{Y} \cap \mathbf{y}) \\
&= \prod\limits_{X \in \mathbf{X}} q(x \mid \painvalues{\mathcal{G}}{X}) = \prod\limits_{X \in \mathbf{X}} r_1(x \mid \painvalues{\mathcal{G}}{X}) = r_1(\mathbf{x})
\end{align*}
Therefore, $r_1(\mathbf{x} \setminus \mathbf{s} \mid \hat{\mathbf{s}}) = q(\mathbf{x} \setminus \mathbf{s} \mid \hat{\mathbf{s}}) = p(\mathbf{x} \setminus \mathbf{s}) = p_1(\mathbf{x} \setminus \mathbf{s})$. Thus, $P_1 \in \bnmodel(\mathcal{G}_1)[^{\mathbf{S} = \hat{\mathbf{s}}}$. 

Let $R_2 \in \bnmodel(\mathcal{G}_2)$ such that $r_2(y \mid \painvalues{\mathcal{G}_2}{Y}) = q(y \mid \painvalues{\mathcal{G}}{Y})$ for each $Y \in \mathbf{Y}$. Then
\begin{align*}
p_2(\mathbf{y} \mid \mathbf{x} \setminus \mathbf{s}) &= p(\mathbf{y} \mid \mathbf{x} \setminus \mathbf{s}) = q(\mathbf{y} \mid \mathbf{x} \setminus \mathbf{s}, \hat{\mathbf{s}}) = \frac{q(\mathbf{o}, \hat{\mathbf{s}})}{q(\mathbf{x} \setminus \mathbf{s}, \hat{\mathbf{s}})} = \prod_{Y \in \mathbf{Y}} q(y \mid \painvalues{\mathcal{G}}{Y}) \\
&= \prod_{Y \in \mathbf{Y}} r_2(y \mid \painvalues{\mathcal{G}_2}{Y}) = r_2(\mathbf{y} \mid \mathbf{x} \setminus \mathbf{s})
\end{align*}
Thus, $P_2 \in \bnmodel(\mathcal{G}_2)$.

\emph{Reverse direction:} Suppose $P_1 \in \bnmodel(\mathcal{G}_1)[^{\mathbf{S} = \hat{\mathbf{s}}}$ and $P_2 \in \bnmodel(\mathcal{G}_2)$. Then there exists $Q_1 \in \bnmodel(\mathcal{G}_1)$ such that $Q_1[^{\mathbf{S} = \hat{\mathbf{s}}} = P_1$. Let $R \in \bnmodel(\mathcal{G})$ such that $r(x \mid \painvalues{\mathcal{G}}{X}) = q_1(x \mid \painvalues{\mathcal{G}_1}{X})$ for each $X \in \mathbf{X}$ and $r(y \mid \painvalues{\mathcal{G}}{Y}) = p_2(y \mid \painvalues{\mathcal{G}_2}{Y})$ for each $Y \in \mathbf{Y}$. Then
\begin{align*}
r(\mathbf{o}, \mathbf{s}) &= \prod\limits_{X \in \mathbf{X}} r(x \mid \painvalues{\mathcal{G}}{X}) \prod\limits_{Y \in \mathbf{Y}} r(y \mid \painvalues{\mathcal{G}}{Y}) \\
&= \prod_{X \in \mathbf{X}} q_1(x \mid \painvalues{\mathcal{G}_1}{X}) \prod_{Y \in \mathbf{Y}} p_2(y \mid \painvalues{\mathcal{G}_2}{Y}) = q_1(\mathbf{x}) p_2(\mathbf{y} \mid \mathbf{x} \setminus \mathbf{s}) 
\end{align*}
and
\begin{align*}
r(\hat{\mathbf{s}}) &= \sum_{\mathbf{o}} r(\mathbf{o}, \hat{\mathbf{s}}) = \sum_{\mathbf{x} \setminus \mathbf{s}} q_1(\mathbf{x} \setminus \mathbf{s}, \hat{\mathbf{s}}) \sum_{\mathbf{y}} p_2(\mathbf{y} \mid \mathbf{x} \setminus \mathbf{s}) = q_1(\hat{\mathbf{s}}) > 0
\end{align*}
Therefore, 
 \begin{align*}
r(\mathbf{o} \mid \hat{\mathbf{s}}) &= \frac{r(\mathbf{o}, \hat{\mathbf{s}})}{r(\hat{\mathbf{s}})} = \frac{q_1(\mathbf{x} \setminus \mathbf{s}, \hat{\mathbf{s}})p_2(\mathbf{y} \mid \mathbf{x} \setminus \mathbf{s})}{q_1(\hat{\mathbf{s}})} = q_1(\mathbf{x} \setminus \mathbf{s} \mid \hat{\mathbf{s}}) p_2(\mathbf{y} \mid \mathbf{x} \setminus \mathbf{s}) \\
&= p_1(\mathbf{x} \setminus \mathbf{s})p_2(\mathbf{y} \mid \mathbf{x} \setminus \mathbf{s}) = p(\mathbf{x} \setminus \mathbf{s})p(\mathbf{y} \mid \mathbf{x} \setminus \mathbf{s}) = p(\mathbf{o})
\end{align*}
Thus, $P \in \bnmodel(\mathcal{G})[^{\mathbf{S} = \hat{\mathbf{s}}}$.
\end{proofln}

Let  $\mathcal{G}$ be a directed graph over $\mathbf{V}$ and $\mathbf{X} \subseteq \mathbf{V}$. $\mathbf{X}$ is \emph{ancestral} if $\anin{\mathcal{G}}{\mathbf{X}} \subseteq \mathbf{X}$. Clearly, $\mathbf{X}$ is ancestral if and only if there are no edges from $\mathbf{V} \setminus \mathbf{X}$ to $\mathbf{X}$.

\begin{lmm}
\label{lmm:ancestral_set_intersection}
Let $\mathbf{G}$ be a Markov equivalence class of \glspl{dag} over $\mathbf{V}$, $\{\mathcal{G}_1,\mathcal{G}_2\} \subseteq \mathbf{G}$, and $\mathbf{A}_1$ and $\mathbf{A}_2$ be ancestral sets in $\mathcal{G}_1$ and $\mathcal{G}_2$, respectively. There exists $\mathcal{G} \in \mathbf{G}$ such that $\mathbf{A}_1 \cup \mathbf{A}_2$ is ancestral in $\mathcal{G}$.
\end{lmm}
\begin{proof}
Let $\mathcal{G}$ be the graph over $\mathbf{V}$ with edges between members of $\mathbf{A}_2$ taken from $\mathcal{G}_1$, and edges between members of $\mathbf{V} \setminus \mathbf{A}_2$ or between a member of $\mathbf{A}_2$ and a member of $\mathbf{V} \setminus \mathbf{A}_2$ taken from $\mathcal{G}_2$. Clearly, $\mathcal{G}$ has the same skeleton as $\mathcal{G}_1$ and $\mathcal{G}_2$.

Suppose that $c$ is a directed cycle in $\mathcal{G}$. If all nodes in $c$ were in $\mathbf{A}_2$ (resp., $\mathbf{V} \setminus \mathbf{A}_2$), then $c$ would be a directed cycle in $\mathcal{G}_1$ (resp., $\mathcal{G}_2$). Therefore, there exists a node in $c$ that is in $\mathbf{A}_2$ and one that is in $\mathbf{V} \setminus \mathbf{A}_2$, which implies that there is an edge $X \rightarrow Y$ in $c$ such that $X \in \mathbf{V} \setminus \mathbf{A}_2$ and $Y \in \mathbf{A}_2$. This is a contradiction, as $X \rightarrow Y$ is taken from $\mathcal{G}_2$ and $\mathbf{A}_2$ is ancestral in $\mathcal{G}_2$. Therefore, $\mathcal{G}$ is acyclic.

Let $(X, Y, Z)$ be an unshielded triple in $\mathcal{G}$. If both edges in the triple are taken from either $\mathcal{G}_1$ or $\mathcal{G}_2$, then $(X, Y, Z)$ is a collider in $\mathcal{G}$ if and only if it is a collider in $\mathcal{G}_1$ and $\mathcal{G}_2$. Suppose that the first edge is taken from $\mathcal{G}_1$ and the second from $\mathcal{G}_2$, which implies that $\{X, Y\} \in \mathbf{A}_2$ and $Z \in \mathbf{V} \setminus \mathbf{A}_2$. Since  $\mathbf{A}_2$ is ancestral in $\mathcal{G}_2$, the second edge is out of $Y$. Therefore, $(X, Y, Z)$ is a noncollider in $\mathcal{G}$, $\mathcal{G}_1$, and $\mathcal{G}_2$. Thus, $\mathcal{G} \in \mathbf{G}$. 

Suppose that there is an edge between $X \in \mathbf{A}_1 \cup \mathbf{A}_2$ and $Y \in \mathbf{V} \setminus (\mathbf{A}_1 \cup \mathbf{A}_2)$ in $\mathcal{G}$. If $Y \in \mathbf{V} \setminus \mathbf{A}_2$, then the edge is taken from $\mathcal{G}_2$ and is out of $X$ because $\mathbf{A}_2$ is ancestral in $\mathcal{G}_2$. Otherwise $Y \in \mathbf{A}_2 \setminus \mathbf{A}_1$ and the edge is taken from $\mathcal{G}_1$ and is out of $X$ because $\mathbf{A}_1$ is ancestral in $\mathcal{G}_1$. Therefore, $\mathbf{A}_1 \cup \mathbf{A}_2$ is ancestral in $\mathcal{G}$.

\end{proof}

\begin{proof}[Proof of Theorem \ref{thm:compelled_ancestor_dag_edges}]
Suppose that $\mathbf{G} = \{\mathcal{G}_1, \ldots, \mathcal{G}_i, \ldots, \mathcal{G}_n\}$ and $\mathbf{A}_i = \anin{\mathcal{G}_i}{\mathbf{Y}}$. It will be proved by induction on $i$ that, for each $i$, there exists $\mathcal{G}'_i \in \mathbf{G}$ such that $\mathbf{A}_1 \cap \ldots \cap \mathbf{A}_i$ is ancestral in $\mathcal{G}'_i$.

\emph{Base case:} Since $\mathbf{A}_1$ is ancestral in $\mathcal{G}_1$, $\mathcal{G}'_1 = \mathcal{G}_1$.

\emph{Inductive step:} Suppose that there exists $\mathcal{G}'_i \in \mathbf{G}$ such that $\mathbf{A}_1 \cap \ldots \cap \mathbf{A}_i$ is ancestral in $\mathcal{G}'_i$. Since  $\mathbf{A}_{i + 1}$ is ancestral in $\mathcal{G}_{i + 1}$, Lemma \ref{lmm:ancestral_set_intersection} implies that there exists $\mathcal{G}'_{i + 1} \in \mathbf{G}$ such that $\mathbf{A}_1 \cap \ldots \cap \mathbf{A}_{i + 1}$ is ancestral in $\mathcal{G}'_{i + 1}$.

Therefore, there exists $\mathcal{G}'_n \in \mathbf{G}$ such that $\mathbf{A}_1 \cap \ldots \cap \mathbf{A}_n = \companin{\mathbf{G}}{\mathbf{Y}}$ is ancestral in $\mathcal{G}'_n$, which means $\mathbf{A}_n = \companin{\mathbf{G}}{\mathbf{Y}}$.

\end{proof}

\begin{proofln}[Proof of Lemma \ref{lmm:compelled_ancestor_dag_condition}]
\emph{Forward direction:} Suppose that $\anin{\mathcal{G}}{\mathbf{Y}} = \mathbf{Z}$. If there is an edge $X \rightarrow Z$ such that $X \in \mathbf{V} \setminus \mathbf{Z}$ and $Z \in \mathbf{Z}$ in $\mathcal{G}$, then $X \in \anin{\mathcal{G}}{\mathbf{Y}}$. This is a contradiction. Therefore, every edge between $X \in \mathbf{V} \setminus \mathbf{Z}$ and $Z \in \mathbf{Z}$ is out of $Z$.

\emph{Reverse direction:} Suppose that every edge between $X \in \mathbf{V} \setminus \mathbf{Z}$ and $Z \in \mathbf{Z}$ is out of $Z$. Since directed paths from $\mathbf{V} \setminus \mathbf{Z}$ to $\mathbf{Y}$ must go through $\mathbf{Z}$, there are no directed paths from $\mathbf{V} \setminus \mathbf{Z}$ to $\mathbf{Y}$ in $\mathbf{G}$. Thus, $\anin{\mathcal{G}}{\mathbf{Y}} = \mathbf{Z}$.
\end{proofln}

\begin{lmm}[\citet{meek1995causal}, Lemma 1]
\label{lmm:directed_edge_undirected_edge} Let $\mathcal{P}$ be a \gls{cpdag}. If triple $X \rightarrow Y \undirectededge Z$ exists in $\mathcal{P}$, then edge $X \rightarrow Z$ exists in $\mathcal{P}$.
\end{lmm}

A simple path from $X$ to $Y$ where all directed edges are directed towards $Y$ is called \emph{possibly directed}.

\begin{lmm}
\label{lmm:updp_out_of_x} Let $p$ be an unshielded possibly directed path from $X$ to $Y$ in a \gls{cpdag}. If $p$ is out of $X$, then $p$ is directed.
\end{lmm}
\begin{proof}
It will be proved by induction on the number of nodes on a subpath of $p$ starting from $X$ that every such subpath is directed.

\emph{Base case:} The result follows by the hypothesis.

\emph{Inductive step:} Suppose that $p(X, V)$ is a directed path. Let $U$ be the predecessor and $W$ be the successor of $V$ on $p$. If the edge between $V$ and $W$ is undirected, then Lemma \ref{lmm:directed_edge_undirected_edge} says that edge $U \rightarrow W$ exists. Since $p$ is unshielded, this is a contradiction. Therefore, the edge between $V$ and $W$ is directed towards $W$.

\end{proof}

\begin{crl}
\label{crl:updp_form} In a \gls{cpdag}, an unshielded possibly directed path from $X$ to $Y$ takes one of three forms:
\begin{enumerate}
	\item $X \rightarrow \cdots \rightarrow Y$
	\item $X \undirectededge \cdots \undirectededge Y$
	\item $X \undirectededge \cdots \undirectededge \rightarrow \cdots \rightarrow Y$
\end{enumerate}
\end{crl}
\begin{proof}
Let $p$ be an unshielded possibly directed path from $X$ to $Y$ in a \gls{cpdag}. If $p$ is not undirected, let $A$ is the first node on $p$ such that $p(A, Y)$ is out of $A$. Owing to Lemma \ref{lmm:updp_out_of_x}, $p(A, Y)$ is directed.
\end{proof}

Let $p_1 = (X_1, \ldots, X_k)$ and $p_2 = (X_k, \ldots, X_{k + n - 1})$ be two paths in graph $\mathcal{G}$. Path $(X_1, \ldots, X_{k + n - 1})$ in $\mathcal{G}$ is the \emph{concatenation} of $p_1$ and $p_2$ and is denoted by $p_1 \oplus p_2$. 

\begin{lmm}
\label{lmm:dp_implies_unshielded_dp} Let $\mathcal{G}$ be a \gls{dag} and $p$ be a directed path from $X$ to $Y$ in $\mathcal{G}$. Then there exists an unshielded directed path from $X$ to $Y$ in $\mathcal{G}$ that goes through a subset of the nodes on $p$.
\end{lmm}
\begin{proof} Suppose that $p$ is shielded and let $(U, V, W)$ be the first shielded triple on $p$. The edge between $U$ and $W$ in $\mathcal{G}$ is out of $U$, because otherwise $(U, V, W)$ would be a directed cycle in $\mathcal{G}$. Therefore, $p' = p(X, U) \oplus p(W, Y)$ is a directed path from $X$ to $Y$ in $\mathcal{G}$. Repeating the same procedure with $p'$ results in an unshielded directed path from $X$ to $Y$ in $\mathcal{G}$ that goes through a subset of the nodes on $p$.
\end{proof}

In a graph $\mathcal{G}$, an edge between two nonconsecutive nodes on a simple cycle is called a \emph{chord}. $\mathcal{G}$ is \emph{chordal} if every simple cycle with four or more distinct nodes has a chord.
Let $\mathbf{G}$ be a Markov equivalence class of \glspl{dag}, $\mathcal{P}$ be its \gls{cpdag}, and $\mathcal{U}$ be the subgraph of $\mathcal{P}$ containing only the undirected edges of $\mathcal{P}$. $\mathcal{U}$ is chordal \citep[see][proof of Theorem 3]{meek1995causal}.

A \emph{clique} in $\mathcal{G}$ is a set of nodes that are all adjacent to each other; a \emph{maximal clique} is a clique that is not contained in another.  A \emph{join tree} $\mathcal{T}$ for graph $\mathcal{G}$ is a tree over the maximal cliques of $\mathcal{G}$ such that, for every pair of maximal cliques $\{\mathbf{M}_1, \mathbf{M}_2\}$, $X \in \mathbf{M}_1 \cap \mathbf{M}_2$, and $\mathbf{M}$ on the simple path from $\mathbf{M}_1$ to $\mathbf{M}_2$ in $\mathcal{T}$, $X \in \mathbf{M}$. $\mathcal{G}$ has a join tree if and only if $\mathcal{G}$ is chordal \citep{beeri1983desirability}.

A total order $<$ on the nodes of undirected graph $\mathcal{U}$ induces an orientation of $\mathcal{U}$ into a directed graph $\mathcal{G}$: if edge $X \undirectededge Y$ exists in $\mathcal{G}$, orient the edge as $X \rightarrow Y$ if $X < Y$ \citep{meek1995causal}. Clearly, $\mathcal{G}$ is acyclic. $<$ is \emph{consistent} with respect to $\mathcal{U}$ if $\mathcal{G}$ has no unshielded colliders.

A partial order $\pi$ is a \emph{tree order} for tree $\mathcal{T}$ if adjacent nodes in $T$ are comparable under $\pi$ \citep{meek1995causal}. Any tree order for $\mathcal{T}$ can be obtained by choosing a root for $\mathcal{T}$ and ordering the nodes based on their distance from the root. Tree order $\pi$ for join tree $\mathcal{T}$ for graph $\mathcal{G}$ induces a partial order $\prec_{\pi}$ on the nodes of $\mathcal{G}$ \citep{meek1995causal}: if $\pi(\mathbf{M}_1, \mathbf{M}_2)$, then for all $X \in \mathbf{M}_1 \cap \mathbf{M}_2$ and $Y \in \mathbf{M}_2$, $X \prec_{\pi} Y$.\footnote{Our definition of $\prec_{\pi}$ is different than the one of \citet{meek1995causal}, as it can be shown that according to the latter, $\prec_{\pi}$ is not partial order. However, the results of \citet{meek1995causal} still hold under our definition of $\prec_{\pi}$.} If $X$ and $Y$ are both in the minimal element of $\pi$, then they are not comparable under $\prec_{\pi}$.

\begin{lmm}[\citet{meek1995causal}, Lemma 4]
\label{lmm:consistent_order} Let $\mathcal{U}$ be a chordal graph, $\mathcal{T}$ be a join tree for $\mathcal{U}$, and $\pi$ be a tree order for $\mathcal{T}$. Any extension of $\prec_{\pi}$ to a total order is a consistent order with respect to $\mathcal{U}$.
\end{lmm}

\begin{lmm}
\label{lmm:chordal_graph_simple_cycle_unshielded_triple} Let $\mathcal{U}$ be a chordal graph and $c$ be a simple cycle in $\mathcal{U}$. There exist two shielded triples on $c$.
\end{lmm}
\begin{proof}
Let $c = (X_1, \ldots, X_n, X_1)$. The result will be proved by induction on the number of edges $n$ of $c$.

\emph{Base case:} If $n = 3$, then $(X_1, X_2, X_3)$ and $(X_2, X_3, X_1)$ are shielded triples on $c$.

\emph{Inductive step:} Suppose that $n \ge 4$ and the result holds for simple cycles with up to $n - 1$ edges. Since $\mathcal{U}$ is chordal, $c$ has a chord $X_i \undirectededge X_j$. By the induction hypothesis, there exist two shielded triples $t_1$ and $t_2$ on $(X_i, \ldots, X_j, X_i)$ and two shielded triples $t_3$ and $t_4$ on $(X_i, X_j, \ldots, X_n, X_1, \ldots, X_i)$. Even if $t_2 = (X_{j - 1}, X_j, X_i)$ and $t_3 = (X_i, X_j, X_{j + 1})$, $t_1$ and $t_4$ are distinct shielded triples on $c$.

\end{proof}

Let $p = (X_1, \ldots, X_n)$ be a path. Path $(X_n, \ldots, X_1)$ is the \emph{reverse} of $p$ and is denoted by $p^{-1}$. A simple cycle is shielded if there is a shielded triple on the cycle.

\begin{proofln}[Proof of Theorem \ref{thm:compelled_ancestor}]
\emph{Forward direction:} Suppose $X \in \anin{\mathcal{P}}{\mathbf{Y}}$ and let $\mathcal{G} \in \mathbf{G}$. There exists a directed path from $X$ to $\mathbf{Y}$ in $\mathcal{G}$. Owing to Lemma \ref{lmm:dp_implies_unshielded_dp}, there exists an unshielded directed path from $X$ to $\mathbf{Y}$ in $\mathcal{G}$. Therefore, there exists an unshielded possibly directed path $p$ from $X$ to $\mathbf{Y}$ in $\mathcal{P}$. Suppose $X \notin \anin{\mathcal{P}}{\mathbf{Y}}$. Then no unshielded possibly directed path from $X$ to $\mathbf{Y}$ in $\mathcal{P}$ is out of $X$ due to Lemma \ref{lmm:updp_out_of_x}. Let $Z$ be the successor $X$ on $p$ and $\mathcal{G} \in \mathbf{G}$ such that edge $X \undirectededge Z$ is oriented as $X \leftarrow Z$ in $\mathcal{G}$. There exists a directed path from $X$ to $\mathbf{Y}$ that does not go through $Z$ in $\mathcal{G}$, because otherwise a directed cycle would occur in $\mathcal{G}$. Lemma \ref{lmm:dp_implies_unshielded_dp} then implies that there exists an unshielded directed path from $X$ to $\mathbf{Y}$ that does not go through $Z$ in $\mathcal{G}$. Therefore, there exists an unshielded possibly directed path from $X$ to $\mathbf{Y}$ that does not go through $Z$ in $\mathcal{P}$.

Suppose that for each pair of unshielded possibly directed paths $(X, Z_1, \ldots, Y_1)$ and $(X, Z_2, \ldots, Y_2)$ in $\mathcal{P}$ such that $Z_1 \neq Z_2$, $Y_1 \in \mathbf{Y}$, and $Y_2 \in \mathbf{Y}$, $Z_1$ and $Z_2$ are adjacent in $\mathcal{P}$. Let $\mathcal{U}$ be the subgraph of $\mathcal{P}$ containing only the undirected edges of $\mathcal{P}$ and $\mathbf{Z}$ be the set of nodes $Z$ such that an unshielded possibly directed path $(X, Z, \ldots, Y)$ ($Y \in \mathbf{Y}$) exists in $\mathcal{P}$. $\{X\} \cup \mathbf{Z}$ is a clique in $\mathcal{U}$; therefore, it is contained in some maximal clique $\mathbf{M}$. Let $\mathcal{T}$ be a join tree for $\mathcal{U}$ with $\mathbf{M}$ as the root, $\pi$ be a tree order for $\mathcal{T}$, $\prec'_{\pi}$ be the extension of $\prec_{\pi}$ such that $Z \prec'_{\pi} X$ for each $Z \in \mathbf{Z}$, and $\alpha$ be a total order which extends $\prec'_{\pi}$. Owing to Lemma \ref{lmm:consistent_order}, $\alpha$ is a consistent order with respect to $\mathcal{U}$. Let $\mathcal{G} \in \mathbf{G}$ be such that $\mathcal{U}$ is oriented according to $\alpha$. There exists a directed path from $X$ to $Y \in \mathbf{Y}$ in $\mathcal{G}$. Owing to Lemma \ref{lmm:dp_implies_unshielded_dp}, there exists an unshielded directed path from $X$ to $Y$ in $\mathcal{G}$. Therefore, there exists an unshielded possibly directed path from $X$ to $Y$ in $\mathcal{P}$ that is out of $X$ in $\mathcal{G}$. This is a contradiction. Thus, there exists a pair of unshielded possibly directed paths $p_1 = (X, Z_1, \ldots, Y_1)$ and $p_2 = (X, Z_2, \ldots, Y_2)$ in $\mathcal{P}$ such that $Z_1 \neq Z_2$, $Y_1 \in \mathbf{Y}$, $Y_2 \in \mathbf{Y}$, and $\nonadjacencyin{\mathcal{P}}{Z_1}{Z_2}$.

Owing to Corollary \ref{crl:updp_form}, $p_1$ is of the form $X \undirectededge Z_1 \cdots \undirectededge \rightarrow \cdots \rightarrow Y_1$ and $p_2$ is of the form $X \undirectededge Z_2 \cdots \undirectededge \rightarrow \cdots \rightarrow Y_2$. Let $U_1$ be the first ancestor of $Y_1$ on $p_1$ (which may be $Y_1$ itself), $U_2$ be the first ancestor of $Y_2$ on $p_2$ (which may be $Y_2$ itself), and $s = p_1(X, U_1)^{-1} \oplus p_2(X, U_2)$. Suppose that a subpath $s'$ of $s$ is a simple cycle. Lemma \ref{lmm:chordal_graph_simple_cycle_unshielded_triple} says that two shielded triples exist on $s'$, which is a contradiction. Therefore, $s$ is a simple path. Thus, $X$ is on an unshielded undirected path in $\mathcal{P}$ between two members of $\anin{\mathcal{P}}{\mathbf{Y}}$ in $\mathcal{P}$.

\emph{Reverse direction:} Suppose $X \notin \anin{\mathcal{P}}{\mathbf{Y}}$ and that $X$ is an interior node on an unshielded undirected path $p$ from $Z_1 \in \anin{\mathcal{P}}{\mathbf{Y}}$ to $Z_2 \in \anin{\mathcal{P}}{\mathbf{Y}}$ in $\mathcal{P}$. Let $U_1$ and $U_2$ be the nodes on either side of $X$ on $p$ and $\mathcal{G} \in \mathbf{G}$. Suppose that $U_1 \undirectededge X$ in $\mathcal{P}$ is oriented as $U_1 \leftarrow X$ in $\mathcal{G}$. Then $p(Z_1, X)^{-1}$ is a directed path in $\mathcal{G}$, because otherwise there would be unshielded colliders in $\mathcal{G}$ that are not in $\mathcal{P}$. Suppose that $U_1 \undirectededge X$ in $\mathcal{P}$ is oriented as $U_1 \rightarrow X$ in $\mathcal{G}$. Then edge $X \undirectededge U_2$ in $\mathcal{P}$ is oriented as $X \rightarrow U_2$ in $\mathcal{G}$ and $p(X, Z_2)$ is a directed path in $\mathcal{G}$,  because otherwise there would be unshielded colliders in $\mathcal{G}$ that are not in $\mathcal{P}$. Therefore, $X \in \anin{\mathcal{G}}{\mathbf{Y}}$.
\end{proofln}

\begin{proof}[Proof of Lemma \ref{lmm:u_up_subgraph_is_tree}]
Let $\mathcal{T}$ be the subgraph of $\mathcal{P}$ containing only the nodes and the edges on unshielded undirected paths from $X$ to $\mathbf{V} \setminus \{X\}$ in $\mathcal{P}$. Clearly, $\mathcal{T}$ is a connected undirected graph. Suppose that there is a simple cycle $c$ in $\mathcal{T}$. Owing to Lemma \ref{lmm:chordal_graph_simple_cycle_unshielded_triple}, there exist two shielded triples on $c$, which is a contradiction. Therefore, $\mathcal{T}$ is a tree.
\end{proof}

\begin{proofln}[Proof of Theorem \ref{thm:selection_conditioning}]
\emph{Forward direction:} Suppose $P \in \bnmodel(\mathcal{G})[^{S = \hat{s}}$. Then there exists $Q \in \bnmodel(\mathcal{G})$ such that $Q[^{S = \hat{s}} = P$. Since $p(\painvalues{\mathcal{G}}{S}) > 0$, $q(\painvalues{\mathcal{G}}{S}, \hat{s}) > 0$. Let $R \in \bnmodel(\mathcal{G}_{\mathbf{O}})$ such that $r(X \mid \painvalues{\mathcal{G}_{\mathbf{O}}}{X}) = q(X \mid \painvalues{\mathcal{G}}{X})$ for each $X \in \mathbf{O}$. Then
\begin{align*}
	p(\mathbf{o} \setminus \painvalues{\mathcal{G}}{S} \mid \painvalues{\mathcal{G}}{S}) &= q(\mathbf{o} \setminus \painvalues{\mathcal{G}}{S} \mid \painvalues{\mathcal{G}}{S}, \hat{s}) = \frac{q(\mathbf{o}, \hat{s})}{q(\painvalues{\mathcal{G}}{S}, \hat{s})} \\
	&= \frac{q(\hat{s} \mid \painvalues{\mathcal{G}}{S}) \prod_{O \in \mathbf{O}} q(o \mid \painvalues{\mathcal{G}}{O})}{q(\hat{s} \mid \painvalues{\mathcal{G}}{S}) \sum_{\mathbf{o} \setminus \painvalues{\mathcal{G}}{S}} \prod_{O \in \mathbf{O}} q(o \mid \painvalues{\mathcal{G}}{O})} \\
	&= \frac{\prod_{O \in \mathbf{O}} q(o \mid \painvalues{\mathcal{G}}{O})}{\sum_{\mathbf{o} \setminus \painvalues{\mathcal{G}}{S}} \prod_{O \in \mathbf{O}} q(o \mid \painvalues{\mathcal{G}}{O})} \\
	&= \frac{\prod_{O \in \mathbf{O}} r(o \mid \painvalues{\mathcal{G}_{\mathbf{O}}}{O})}{\sum_{\mathbf{o} \setminus \painvalues{\mathcal{G}}{S}} \prod_{O \in \mathbf{O}} r(o \mid \painvalues{\mathcal{G}_{\mathbf{O}}}{O})} \\
	&= \frac{r(\mathbf{o})}{r(\painvalues{\mathcal{G}}{S})} = r(\mathbf{o} \setminus \painvalues{\mathcal{G}}{S} \mid \painvalues{\mathcal{G}}{S})
\end{align*}

Therefore, $P[^{\pain{\mathcal{G}}{S}} \in \bnmodel(\mathcal{G}_{\mathbf{O}})[^{\pain{\mathcal{G}}{S}}$.
	
\emph{Reverse direction:} Suppose $P[^{\pain{\mathcal{G}}{S}} \in \bnmodel(\mathcal{G}_{\mathbf{O}})[^{\pain{\mathcal{G}}{S}}$. Therefore, there exists $Q \in \bnmodel(\mathcal{G}_{\mathbf{O}})$ such that $Q[^{\pain{\mathcal{G}}{S}} = P[^{\pain{\mathcal{G}}{S}}$. Let $R \in \bnmodel(\mathcal{G})$ such that $r(o \mid \painvalues{\mathcal{G}}{O}) = q(o \mid \painvalues{\mathcal{G}_{\mathbf{O}}}{O})$ for each $O \in \mathbf{O}$ and $r(\hat{s} \mid \painvalues{\mathcal{G}}{S}) = c \cdot p(\painvalues{\mathcal{G}}{S}) / q(\painvalues{\mathcal{G}}{S})$, where $c = 1 / \max_{\painvalues{\mathcal{G}}{S}}(p(\painvalues{\mathcal{G}}{S})/q(\painvalues{\mathcal{G}}{S}))$ (ensuring that $r(\hat{s} \mid \painvalues{\mathcal{G}}{S}) \le 1$). Then
\begin{align*}
r(\mathbf{o}, \hat{s}) &= r(\hat{s} \mid \painvalues{\mathcal{G}}{S}) \prod_{O \in \mathbf{O}} r(o \mid \painvalues{\mathcal{G}}{O}) = r(\hat{s} \mid \painvalues{\mathcal{G}}{S}) \prod_{O \in \mathbf{O}} q(o \mid \painvalues{\mathcal{G}_{\mathbf{O}}}{O}) \\
&= r(\hat{s} \mid \painvalues{\mathcal{G}}{S}) q(\mathbf{o}) = r(\hat{s} \mid \painvalues{\mathcal{G}}{S}) q(\painvalues{\mathcal{G}}{S}) q(\mathbf{o} \setminus \painvalues{\mathcal{G}}{S} \mid \painvalues{\mathcal{G}}{S}) \\
&= c \cdot \frac{p(\painvalues{\mathcal{G}}{S})}{q(\painvalues{\mathcal{G}}{S}} q(\painvalues{\mathcal{G}}{S}) p(\mathbf{o} \setminus \painvalues{\mathcal{G}}{S} \mid \painvalues{\mathcal{G}}{S}) = c \cdot p(\mathbf{o})
\end{align*}
and $r(\hat{s}) = \sum_{\mathbf{o}} r(\mathbf{o}, \hat{s}) = c \sum_{\mathbf{o}} p(\mathbf{o}) = c > 0$. Therefore, 
\begin{align*}
&r(\mathbf{o} \mid \hat{s}) = \frac{r(\mathbf{o}, \hat{s})}{r(\hat{s})} = \frac{c \cdot p(\mathbf{o})}{c} = p(\mathbf{o})
\end{align*}
Thus, $P \in \bnmodel(\mathcal{G})[^{S = \hat{s}}$.
\end{proofln}

\bibliographystyle{plainnat}
\bibliography{../../Common/refs}

\end{document}